\documentclass{amsart}

\usepackage{amssymb,mathrsfs,amsmath,amsthm,color,bm,mathtools,bbm,wasysym,cases,mathdots}
\usepackage{xcolor}
\usepackage[normalem]{ulem}
\usepackage[pagebackref]{hyperref}
\usepackage[noadjust]{cite}

\usepackage{enumerate}
\usepackage[centering]{geometry}
\newtheorem{theorem}{Theorem}[section]
\newtheorem{lemma}{Lemma}[section]
\newtheorem{proposition}{Proposition}[section]

\theoremstyle{definition}
\newtheorem{definition}{Definition}
\newtheorem{remark}{Remark}

\numberwithin{equation}{section}

\newcommand{\R}{\mathbb R}
\newcommand{\Q}{\mathbb Q}
\newcommand{\Z}{\mathbb Z}
\newcommand{\N}{\mathbb N}
\newcommand{\T}{\mathbb T}
\newcommand{\Mat}{\mathrm{M}}
\newcommand{\Exact}{\operatorname{Exact}}

\DeclareMathOperator{\hdim}{\dim_H}

\newcommand{\qaq}{\mathrm{\quad and\quad}}
\newcommand{\dif}{{\, \mathrm d}}

\newcommand{\cu}{\mathcal U}

\newcommand{\ba}{\mathbf{a}}
\newcommand{\bb}{\mathbf{b}}
\newcommand{\bk}{\mathbf{k}}

\newcommand{\br}{\mathbf{r}}

\newcommand{\bx}{\mathbf{x}}
\newcommand{\bp}{\mathbf{p}}

\newcommand{\bu}{\mathbf{u}}
\newcommand{\bv}{\mathbf{v}}

\newcommand{\bta}{{\bm\theta}}

\newcommand{\ve}{\varepsilon}

\newcommand{\bw}{\mathbf{w}}

\begin{document}
\title[Exactly approximable matrices]
{Hausdorff dimension for exactly approximable matrices}
	\author[Y. He]{Yubin He}
	\address[Y. He]{Department of Mathematics, Shantou University, Shantou, Guangdong, 515063, China}
	\email{ybhe@stu.edu.cn}

\author[B. Li]{Bixuan Li}
\address[B. Li]{Institute of Mathematics, Academy of Mathematics and System Science, CAS, Zhongguancun East Road 55, Beijing 100190, P. R. China}
	\email{libixuan@amss.ac.cn}

\author[B. Wang]{Baowei Wang}
\address[B. Wang]{School of Mathematics and Statistics, Huazhong University of Science and Technology, Wuhan 430074, P. R. China}
	\email{bwei\_wang@hust.edu.cn}

\subjclass[2020]{11J83, 28A78}

\keywords{Diophantine approximation, exact approximation, inhomogeneous approximation, Hausdorff dimension}

\begin{abstract}
We establish the Hausdorff dimension for exactly approximable matrices for all dimensions, including the inhomogeneous case. More precisely, let $\psi:\N\to(0,\infty)$ be a non-increasing approximation function and let $\tau_\psi$ denote its lower order.  For $m,n\ge1$ and $\boldsymbol\theta\in\mathbb R^m$, denote $\Exact_{m,n}(\psi,\boldsymbol\theta)$ the set of $m\times n$ matrices that are $\psi$-approximable with shift $\boldsymbol\theta$, but are not $c\psi$-approximable for any
$0<c<1$. It is shown that
\[
\hdim \Exact_{m,n}(\psi,\boldsymbol\theta)=m(n-1)+\frac{m+n}{1+\tau_\psi},\ \ \ {\text{for}}\ \tau_\psi>\frac{n}{m}.
\]
We adopt a probabilistic viewpoint, i.e. by an averaged counting argument to quantify the distribution of rational vectors and subspaces
$$\left\{\frac{\mathbf p+\boldsymbol\theta}{q}: (\mathbf p, q)\in \mathbb{Z}^{m+1}\right\},\ \ \
\bigg\{\big\{A\in \Mat_{m, n}(\R): A\mathbf q-\mathbf p-\boldsymbol\theta=0\big\}: (\mathbf p, \mathbf q)\in \mathbb{Z}^{m+n}\bigg\}.$$
\end{abstract}

	\maketitle

\section{Introduction}

Aimed at the distribution of rational numbers and their extensions, Metric Diophantine approximation starts from  studying the measure and
Hausdorff dimension of the $\psi$-well approximable set
$$W(\psi)=\Big\{x\in [0,1]: |x-p/q|<\psi(q), \ {\text{for infinitely many}}\ (p,q)\in \N\times \mathbb{Z}\Big\}$$
by Khintchine \cite{Khintchine} and Jarn\'{i}k \cite{Jar1}.

This foundational framework was subsequently extended to the simultaneous case and even the system of linear forms. We work in the latter
setting and allow an arbitrary inhomogeneous shift.

Fix some notation at first. For $\mathbf x=(x_1,\dots,x_d)\in\R^d$, write
\[
|\mathbf x|:=\max_{1\le i\le d}|x_i|,\ \ \ \|\mathbf x\|:=\min_{\mathbf k\in\Z^d}|\mathbf x-\mathbf k|\]
for the sup norm of $\mathbf{x}$ and the distance of $\mathbf x$ to the integer vectors respectively. Denote by $\Mat_{m,n}(\R)$ the set of real $m\times n$ matrices.

Throughout, let $\psi:\N\rightarrow(0,\infty)$ be a non-increasing function, and fix
$\boldsymbol\theta\in\R^m$. The set of $\psi$-well approximable matrices is defined as
\[W_{m,n}(\psi,\boldsymbol\theta):=\Big\{A\in\Mat_{m,n}(\R):\|A\mathbf q-\boldsymbol\theta\|<\psi(|\mathbf q|)\text{ for infinitely many }
\mathbf q\in\Z^n\setminus\{\mathbf0\}
\Big\}.\]
The metric theory of $W_{m,n}(\psi,\boldsymbol{\theta})$ has been well established in the literature. Khintchine--Groshev theorem shows that its Lebesgue measure satisfies a zero--full law: null or full according as
\[
\sum_{q=1}^{\infty}q^{n-1}\psi(q)^m<\infty, \ {\text{or}}\ =\infty.
\]
The homogeneous case $\boldsymbol\theta=0$ was shown in Khintchine~\cite{Khin2} and Groshev~\cite{Gro} and the
inhomogeneous case was shown in Sprind\v{z}uk~\cite{Spr}. For its Hausdorff dimension, Dodson (homogeneous case \cite{Dodson}) and Levesley (inhomogeneous case \cite{Levesley}) proved that
\begin{equation*}
    \dim_H W_{m,n}(\psi,\boldsymbol{\theta})= m(n-1)+\frac{m+n}{1+\tau_\psi}, \qquad \tau_\psi:=\liminf_{q\to\infty}\frac{-\log\psi(q)}{\log q}>\frac{n}{m}.
\end{equation*}
For Hausdorff measure results, see Dickinson and Velani~\cite{DicV} for the homogeneous case and Bugeaud~\cite{Bug} for the inhomogeneous case.
Further results concerning the removal of monotonicity of $\psi$ can be found in Beresnevich and Velani~\cite{BV10}, Allen and Ram\'irez~\cite{AlR1,AlR2},
Hauke~\cite{Hauke25}, Kim~\cite{Kim25}, and Allen, Hauke and
Ram\'irez~\cite{AHR26}. For a general framework for the metric theory of limsup sets, see Beresnevich, Dickinson and Velani~\cite{Ber}.

The central object of this paper is the Hausdorff dimension of the set of exactly approximable matrices. The study of exact approximation goes back to Jarn\'ik~\cite{Jarnik} in 1931.

\begin{definition}
A matrix $A\in M_{m,n}(\mathbb{R})$ is called \emph{exactly $(\psi,\boldsymbol{\theta})$-approximable} if
\begin{itemize}
\item there are infinitely many pairs $(\mathbf{p},\mathbf{q})\in\mathbb{Z}^m\times (\mathbb{Z}^n\setminus\{\mathbf{0}\})$ satisfying
\[
|A\mathbf{q}-\mathbf{p}-\boldsymbol{\theta}|<\psi(|\mathbf{q}|);
\]

\item for every $c\in(0,1)$, there are only finitely many such pairs satisfying
\[
|A\mathbf{q}-\mathbf{p}-\boldsymbol{\theta}|< c\psi(|\mathbf{q}|).
\]
\end{itemize}
We denote the set of these matrices by $\operatorname{Exact}_{m,n}(\psi,\boldsymbol{\theta})$. Thus,
\[
\operatorname{Exact}_{m,n}(\psi,\boldsymbol{\theta})=W_{m,n}(\psi,\boldsymbol{\theta})\setminus\bigcup_{0<c<1}W_{m,n}(c\psi,\boldsymbol{\theta}).
\]
\end{definition}

By definition, an exactly $(\psi,\boldsymbol{\theta})$-approximable matrix must satisfy two competing requirements: it must lie in the prescribed neighborhoods of infinitely many resonant sets: affine subspaces
\[
A\mathbf{q}-\mathbf{p}-\boldsymbol{\theta}=\mathbf{0};
\]
while, for every fixed $c\in(0,1)$, it must avoid all but finitely many of the corresponding smaller neighborhoods. Compared with the $\psi$-well approximable set $W_{m,n}(\psi,\boldsymbol\theta)$, understanding $\operatorname{Exact}_{m,n}(\psi,\boldsymbol{\theta})$ requires information not only on the density of the resonant sets but also their local separation.

The density property of the resonant sets can be obtained through Dirichlet's theorem: for every $A\in M_{m,n}(\mathbb{R})$, there are infinitely many pairs $(\mathbf{p},\mathbf{q})\in\mathbb{Z}^m\times(\mathbb{Z}^n\setminus\{\mathbf{0}\})$ such that
\[
|A\mathbf{q}-\mathbf{p}|<|\mathbf{q}|^{-n/m}.
\]
This is the essential ingredient for the ubiquity argument in establishing the dimension of $W_{m,n}(\psi,\bold{0})$.

As far as the separation condition is concerned, one has some separation information on the rational numbers ($m=n=1$) and rational vectors \cite{KLWZ} (by simplex lemma $n=1, m\ge 1$). However, for $n\ge 2$ much less is known and no information for the inhomogeneous case even $m=n=1$.

To the best of our knowledge, currently a complete Hausdorff dimension formula for exactly approximation was known only for the case $n=1$ with $\boldsymbol{\theta}=\mathbf{0}$ as cited below.

Under suitable decay assumptions on $\psi$, Jarn\'ik~\cite{Jarnik} showed that $\operatorname{Exact}_{1,1}(\psi,0)$ is non-empty. Bugeaud~\cite{Bug1,Bu2} developed a systematic theory of $\operatorname{Exact}_{1,1}(\psi,0)$ including its Hausdorff dimension. These results were further developed by Bugeaud and Moreira~\cite{BugM}. Tan, Tian and Zhou~\cite{TTZ} and Tan, Tian, Wang and Wu \cite{TTWW} obtained Hausdorff measure for $\operatorname{Exact}_{1,1}(\psi,0)$. Bandi, Ghosh and Nandi~\cite{BanGN} introduced an abstract framework for exact approximation by well-distributed systems in regular metric spaces, which includes the classical one-dimensional setting.

For high dimensional case $m\ge 1, n=1$ and $\boldsymbol{\theta}=\mathbf{0}$, Fregoli~\cite{Fregoli} obtained the dimension for
$m\ge3$ under some restrictions on $\tau_\psi$, for $m=2$, see Tan, Tian and Zhou~\cite{TTZ}. This additional restriction on $\tau_{\psi}$ was later removed by Bandi and de Saxc\'e \cite{BanS} using the theory of parametric geometry of numbers so completing the dimension of $\Exact_{m,1}(\psi,\mathbf0)$:
\[
\hdim\operatorname{Exact}_{m,1}(\psi,\mathbf{0})=\frac{m+1}{1+\tau_\psi}, \  \ {\text{for}}\ \tau_{\psi}>\frac{1}{m}.
\]
The corresponding weighted setting was later established by Bandi and Fregoli \cite{BanF}.

Nothing was known on the Hausdorff dimension of $\operatorname{Exact}_{m,n}(\psi,\mathbf{0})$ for $n\ge 2$ and $\operatorname{Exact}_{m,n}(\psi,\boldsymbol{\theta})$ for all $m\ge 1, n\ge 1$ with $\boldsymbol{\theta}\ne \mathbf 0$. There are obvious differences between the case $n=1$ and the case $n\ge2$ or with arbitrary inhomogeneous shifts.

\begin{enumerate}
\item \emph{Separation of shifted rationals.} When $m=n=1$ and $\theta=0$, distinct rational points satisfy
\[
\left|\frac{p_1}{q_1}-\frac{p_2}{q_2}\right|\ge\frac1{q_1q_2}.
\]
However when $\theta\ne 0$, the shifted rationals can be very close
\[
\left|\frac{p_1+\theta}{q_1}-\frac{p_2+\theta}{q_2}\right|=\frac{|q_2p_1-q_1p_2+(q_2-q_1)\theta|}{q_1q_2}.
\]
So even for this simplest case, one need find new ideas to study it.

    \item
\emph{The geometry of higher-dimensional resonant sets.}
When $n\ge 2$, the resonant sets are hyperplanes
\[
A\mathbf{q}-\mathbf{p}-\boldsymbol{\theta}=\mathbf{0}.
\]
The dangerous sets to be avoided are therefore neighborhoods of affine subspaces, while the separation properties of these subspaces are much unclear.
\end{enumerate}

In this paper, we introduce an average counting argument that quantifies the abundance of well-separated resonant sets. 
In contrast, Bandi and de Saxc\'e~\cite{BanS} rely on a finer analysis on distribution of rational vectors through the Dani correspondence and parametric 
geometry of numbers. The averaged approach allows us to treat the linear forms in all dimensions. We give some rough words on the ideas.

\begin{remark}{\em Case $n=1$.
Define the exact cube and dangerous cube respectively as
 $$
 A(\bold{p}, q):=\frac{\bold{p}+\boldsymbol{\theta}}{q}+\left[\eta\frac{\psi(q)}{q}, \frac{\psi(q)}{q}\right]^m, \ \ \ \ D(\bold{r},s):=B\left(\frac{\bold{r}+\boldsymbol{\theta}}{s}, \eta\frac{\psi(s)}{s}\right).
$$
Theorem \ref{t:main} for $n=1$ case is proved by elementary estimates. A key ingredient we find is
that once the dangerous cubes of all high levels takes only a small proportion inside an exact cube $A(\bold{p}, q)$, called good cube,
then there are sufficiently many good sub-exact cubes inside $A(\bold{p}, q)$ and disjoint from all dangerous cubes in previous levels. This is
achieved by using an average counting argument.

Case $n\ge 2$, we work on fibers $U\in \Mat_{m,n-1}(\mathbb{T})$, which reduces the question similar to the inhomogeneous $n=1$ case.
By integrating on fibers and the similar average counting argument, we show that on most fibers, many exact cubes are good. This shows some inherence property of inhomogeneous Diophantine in dimension for $n\ge 2$ from $n=1$. (There are also inherence property in measure theory as observed by Allen and Ram\'irez~\cite{AlR1} before.)

This average counting method may also be helpful for other questions, especially when the resonant sets have arithmetic properties, but their geometry structure is not clear enough.
}\end{remark}

The main result of this paper is that we determine the Hausdorff dimension of exactly approximable matrices for every $m,n\ge1$ and every shift $\boldsymbol{\theta}\in\mathbb{R}^m$.

\begin{theorem}\label{t:main}
Let $m,n\ge1$, $\boldsymbol{\theta}\in\mathbb{R}^m$, and let $\psi:\mathbb{N}\to(0,\infty)$ be non-increasing.
Then,
\[
\dim_H\operatorname{Exact}_{m,n}(\psi,\boldsymbol{\theta})= m(n-1)+\frac{m+n}{1+\tau_\psi},\ \ \ {\text{for}} \ \tau_\psi>\frac{n}{m}.
\]
\end{theorem}

For related results on exact power order, exact logarithmic order, badly approximable set, and high dimensional Lagrange spectrum, see G\"uting~\cite{Guting}, Beresnevich, Dickinson and Velani~\cite{BerDV}, He and Xiong \cite{HeXiong}, Koivusalo, Levesley, Ward and Zhang~\cite{KLWZ}, Schleischitz~\cite{Schl},
and Ward \cite{Ward}.
\subsection{Notation}

By the invariance of integer translations of the set  $\operatorname{Exact}_{m,n}(\psi,\boldsymbol{\theta})$, we focus the matrixes with entries in unit circle $\mathbb{T}$, i.e. $A\in \Mat_{m,n}(\T)$. Lebesgue measure on $\R^d$ and normalized Haar measure on $\T^d$ are both denoted by $\mathcal L_d$. For one dimensional integer, we use $q$; for high dimensional integer, we use the bold form $\mathbf q$.

For a bounded set $E$ in a metric space, we write $|E|$ for its diameter. For non-negative quantities $X$ and $Y$, we write $X\ll Y$ if $X\le CY$ for some constant $C>0$. If the implied constant is allowed to depend on a parameter $\eta$, we write $X\ll_\eta Y$. The notation $X=O(Y)$ and $X=O_\eta(Y)$ has the same meaning. We write $X\asymp Y$ if both $X\ll Y$ and $Y\ll X$ hold, and use $X\asymp_\eta Y$ when the implied constants may depend on $\eta$. Unless otherwise stated, all implied constants are independent of the main asymptotic parameters.

\subsection{Organization of the paper}

In Section~\ref{s:preliminaries}, we establish some auxiliary results
for later use, including a refined Denjoy--Koksma inequality.
In Section~\ref{s:1-case}-\ref{s:1-proof}, we treat $n=1$ case by analyzing the distribution of shifted rational vectors. In Section~\ref{s:2-case}-\ref{s:2-large}, we treat $n\ge2$.
We work on fibers of the matrix space which reduces the question similar to inhomogeneous $n=1$ case.

\section{Preliminaries}\label{s:preliminaries}

The first lemma shows that for a non-increasing function $\psi$, it can be reasonably large for infinitely many $M$, and for $t\ge M$, the value of $\psi(t)$ can be controlled by $\psi(M)$ with a fixed power of $M/t$.

\begin{lemma}[Regular scale]
\label{lem2.1}
Let $\psi$ be a non-increasing function. For any $0<\epsilon<\tau_{\psi}$, there are arbitrarily large integers $M$ such that
\[\psi(M)\ge M^{-(\tau_\psi+\varepsilon)},\ \ {\text{and}}\ \ \psi(t)\le\psi(M)\left(\frac Mt\right)^{\tau_\psi-\varepsilon}\text{ for any $t\ge M$}.\]
\end{lemma}

\begin{proof}
Recall the definition of $\tau_\psi$. So, on the one hand, for all sufficiently large $t$,
\[
-\frac{\log\psi(t)}{\log t}>\tau_\psi-\frac{\varepsilon}{2}, \ \ {\text{i.e.}}\ \ \psi(t)\le t^{-\tau_{\psi}+\varepsilon/2}.
\]
Then, it follows that
\[
t^{\tau_\psi-\varepsilon}\psi(t)\le t^{-\varepsilon/2}\rightarrow 0.
\]

On the other hand, there are arbitrarily large integers $N$ such that
\[
\psi(N)\ge N^{-(\tau_\psi+\varepsilon/2)}.
\]
Fix such an integer $N$.  Since $t^{\tau_\psi-\varepsilon}\psi(t)\rightarrow0$, this sequence attains its maximum on $\{N,N+1,\ldots\}$. Let $M\ge N$ be a maximizer.  Then, for every $t\ge M$,
\[t^{\tau_\psi-\varepsilon}\psi(t)\le M^{\tau_\psi-\varepsilon}\psi(M), \ {\text{and so}}\ \psi(t)\le \psi(M)\left(\frac Mt\right)^{\tau_\psi-\varepsilon}.
\]

By maximality of the choice of $M$, one has
\[
M^{\tau_\psi-\varepsilon}\psi(M)\ge N^{\tau_\psi-\varepsilon}\psi(N)\ge N^{-3\varepsilon/2}\ge M^{-3\varepsilon/2}.
\]
Thus,
\[
\psi(M)\ge M^{-(\tau_\psi+\varepsilon/2)}\ge M^{-(\tau_\psi+\varepsilon)}.
\]
Since $N$ can be chosen arbitrarily large, so can $M$.
\end{proof}

From now on, we call a real number $M$ a \emph{regular scale} if it satisfies the conclusion of Lemma~\ref{lem2.1}.

\subsection{Denjoy--Koksma inequality}
Let $x\in\R\setminus\Q$, and write $x=[a_0;a_1,a_2,\ldots]$
for its regular continued fraction expansion. Let $\{q_j\}_{j\ge 0}$ be the denominators of the convergents given by
\[
\frac{p_j}{q_j}=[a_0;a_1,\ldots,a_j],\ \ j\ge0.
\]

For the case of $m=1$ with $\theta\notin\Q$, the Denjoy--Koksma inequality will be used.

\begin{lemma}[{\cite[Chapter 2, Theorem 2.5]{KimK}}]\label{dk} Let $\theta\notin\Q$ and let $q_j$ be a denominator of the continued fraction convergent of $\theta$.  For any integers $u, K$ with $K\le q_j$,
and $0<t<1/2$,
\begin{equation*}
\#\big\{u\le h<u+K: \|h\theta\|<t\big\}\le 2(q_j t+1).
\end{equation*}
\end{lemma}

We provide a revised version, since when $K$ is much smaller than $q_j$, the above estimation may not be effective. For an integer $K\ge 1$, define
\begin{equation*}
Q_\theta(K):=\max\{q_j:q_j\le K\}.
\end{equation*}

\begin{lemma}[General Denjoy--Koksma inequality]\label{lem2.3} Let $\theta\notin\Q$ and $0<t<1/2$. For any integers $K\ge 1$ and $u\ge 1$,
\[
\#\big\{u\le h<u+K: \|h\theta\|\le t\big\}\le 4Kt+\frac{4K}{Q_\theta(K)}.
\]
\end{lemma}

\begin{proof}
Since $Q_\theta(K)\le K$, we partition the interval $[u,u+K)$ into consecutive blocks of length at most $Q_\theta(K)$. The number of such blocks is at most
\[\left\lceil\frac{K}{Q_\theta(K)}\right\rceil\le\frac{K}{Q_\theta(K)}+1\le\frac{2K}{Q_\theta(K)}.
\]

On each block, we apply Lemma~\ref{dk} with $M=Q_\theta(K)$ which gives
\[
\#\Big\{u+iQ_\theta(K)\le h< u+(i+1)Q_\theta(K):\|h\theta\|\le t\Big\}\le 2\left(Q_\theta(K)t+1\right).
\]
Hence,
\begin{align*}
\#\Big\{u\le h<u+K: \|h\theta\|\le t\Big\}&\le\frac{2K}{Q_\theta(K)}\left(2Q_\theta(K)t+2\right)=4Kt+\frac{4K}{Q_\theta(K)}.
\end{align*}
This proves the result.
\end{proof}

\begin{remark}
Since $\theta\notin\Q$, the denominators of its continued fraction convergents tend to infinity. Consequently, $Q_\theta(K)\rightarrow\infty$ as $K\to\infty$, and therefore
\[
\frac{K}{Q_\theta(K)}=o(K).
\]
\end{remark}

\subsection{A Cantor structure and its dimension}
We give a framework of Cantor set and its dimension which caters for the applications in our setting.

The Cantor set is defined by a nested family of cubes.  Let
\[
\mathcal F_0=\{J_0\},
\]
where $J_0\subset\T^m$ is a cube.  For each $j\ge1$ and each $P\in\mathcal F_{j-1}$, let $\mathcal F_j(P)$ be a finite family of cubes contained in $P$. These are called the children of $P$, and $P$ is their parent. Define
\[
\mathcal F_j=\bigcup_{P\in\mathcal F_{j-1}}\mathcal F_j(P).
\]
All cubes in $\mathcal F_j$ have the same side length $L_j$ and $L_j\rightarrow0.$ Set
\[
K_j=\bigcup_{Q\in\mathcal F_j}Q, \ \ \ K=\bigcap_{j=0}^{\infty}K_j.
\]

We use the mass distribution principle \cite[Proposition 4.2]{Falconer} to estimate the dimension of $K$.
\begin{proposition}[\cite{Falconer}] Let $E$ be a Borel measurable set and $\mu$ a Borel probability measure supported on $E$. Assume that
\[
\mu(I)\ll |I|^s
\]
for all sufficiently small cubes $I$, then
\[
\hdim E\ge s.
\]\end{proposition}

The dimension of $K$ depends on the separation condition and the number of the cubes in $\mathcal{F}_j(P)$. 
\begin{lemma}\label{l:cantor} Suppose that, for each $j\ge1$, every $P\in\mathcal F_{j-1}$ has exactly $b_j$ children, and write
\[
B_0=1, \ {\text{and}}\ \  B_j=\prod_{i=1}^j b_i, \ \ {\text{for all}}\ j\ge 1.
\]
Let $\{M_j\}$ be a positive sequence with $M_j\to\infty$. Assume that for some $\rho>1$ the following is true.
\begin{enumerate}
\item For each $P\in \mathcal{F}_{j-1}$, the children in $\mathcal F_j(P)$ are separated by a gap $\gg M_j^{-\rho}$.

\item For every cube $I\subset\T^m$ of side length $r<L_{j-1}$ and every $P\in\mathcal F_{j-1}$,
\begin{equation}\label{f2.1}
\#\left\{Q\in\mathcal F_j(P): Q\cap I\ne\varnothing\right\}\ll b_j\left(\frac{r+M_j^{-1}}{L_{j-1}}\right)^m.
\end{equation}
All of the implied constants in $\gg$, $\ll$ above are independent of $j$, $P$, $I$, and $r$.
	
\item For all sufficiently large $j$,
\begin{equation*}
L_j\le M_j^{-\rho} \le r_j:=\frac{b_j^{1/m}M_j^{-(1+\rho)}}{L_{j-1}}\le M_j^{-1}<L_{j-1}.
\end{equation*}
\end{enumerate}
Then,
\begin{equation}\label{dimbound}
\hdim K\ge \min\left\{m,\,\liminf_{j\to\infty}\frac{\log B_j}{-\log L_j},\, \liminf_{j\to\infty}\frac{\log B_{j-1}+m\log(L_{j-1}M_j)}{-\log r_j}\right\}.
\end{equation}
\end{lemma}

\begin{remark}
  We give a remark on the quantity $r_j$. For a cube $I$ of sidelength $r$ and $P\in \mathcal{F}_{j-1}$, there are two ways to bound the number of cubes $Q\in \mathcal{F}_{j}(P)$ which can intersect $I$.
  \begin{itemize}
  \item By volume estimation of cubes $Q$: if $Q\in \mathcal{F}_j(P)$ intersects $I$, it will lie in a neighborhood $I_{\delta}$ of $I$.
      In later applications, this will be the number given in (\ref{f2.1});
  \item By volume estimation of gaps: $M_j^{-\rho}$-separation of the cubes in $\mathcal{F}_{j}(P)$ gives
  $$\#\left\{
			Q\in\mathcal F_j(P): Q\cap I\ne\varnothing\right\}\ll \left(1+rM_j^{-\rho}\right)^m.
			$$
  \end{itemize}
 The quantity $r_j$ is nothing but the boundary of $r$ on which estimation is more effective.
\end{remark}
\begin{proof}
To use the mass distribution principle, we define a probability measure $\mu$ supported on $K$:  let $\mu(J_0)=1$, and distribute the mass of every parent equally among its children. Then each cube $Q\in\mathcal F_j$ has mass
\begin{equation*}
\mu(Q)=B_j^{-1}.
\end{equation*}

Denote the quantity in the right side of (\ref{dimbound}) by $s_0$. If  $s_0\le 0$, there is nothing to prove. So we assume $s_0>0$ and let $s$ be any positive number smaller than $s_0$.  We show that
\[
\mu(I)\ll r^s
\]
for every sufficiently small cube $I$ of side length $r$.

Choose $j$ such that
\[
L_j\le r<L_{j-1}.
\]
Since the level-$(j-1)$ cubes are disjoint and have the same side length $L_{j-1}$, the cube $I$ can meet only $O_m(1)$ of them.  It is therefore enough to consider one parent $P\in\mathcal F_{j-1}$.

By condition (2),
\begin{equation}\label{eq:grid-mass}
\mu(I\cap P)\ll B_{j}^{-1}\cdot b_j\left(\frac{r+M_j^{-1}}{L_{j-1}}\right)^m=B_{j-1}^{-1}\left(\frac{r+M_j^{-1}}{L_{j-1}}\right)^m.
\end{equation}
On the other hand, the separation of the children gives
\[
\#\left\{Q\in\mathcal F_j(P): Q\cap I\ne\varnothing\right\}\ll (1+rM_j^\rho)^m,
\]
and hence
\begin{equation}\label{eq:separation-mass}
\mu(I\cap P)\ll B_j^{-1}(1+rM_j^\rho)^m.	
\end{equation}

We now consider the ranges of $r$ compared with the quantities appearing in condition (3).

	\noindent{\bf Case 1: $M_j^{-1}\le r<L_{j-1}$}. By the choice of $s$,
\[
s<\liminf_{j\to\infty}\frac{\log B_j}{-\log L_j},
\]
we have, for all sufficiently large $j$,
\[
B_{j-1}^{-1}\le L_{j-1}^s.
\]
Therefore, by \eqref{eq:grid-mass},
\[
\mu(I\cap P)\ll B_{j-1}^{-1}\left(\frac{r}{L_{j-1}}\right)^m\ll L_{j-1}^{s-m}r^m\le r^s,
\]
since $s<m$ and $r<L_{j-1}$.

	\noindent{\bf Case 2: $r_j\le r<M_j^{-1}.$} By the choice of $s$,
\begin{align}\label{eq:detail1}
B_{j-1}^{-1}(L_{j-1}M_j)^{-m}\le r_j^s
\end{align}
for all sufficiently large $j$.  Since $r\ge r_j$, \eqref{eq:grid-mass} gives
\[
\mu(I\cap P)\ll B_{j-1}^{-1}(L_{j-1}M_j)^{-m}\ll r^s.
\]

	\noindent{\bf Case 3: $	M_j^{-\rho}\le r<r_j.$} It follows from \eqref{eq:separation-mass} that
\[
\mu(I\cap P)\ll B_j^{-1}\cdot M_j^{m\rho}\cdot r^m\le B_j^{-1}M_j^{m\rho}r_j^{m-s}\cdot r^s,
\]
since $s<m$ and $r<r_j$.  By the definition of $r_j$, it follows that
\[B_j^{-1}M_j^{m\rho}r_j^{m-s}=B_{j-1}^{-1}(L_{j-1}M_j)^{-m}r_j^{-s}.\]
Together with (\ref{eq:detail1}), we have that
\[
\mu(I\cap P)\ll r^s.
\]

	\noindent{\bf Case 4: $L_j\le r<M_j^{-\rho}.$} By the separation assumption (1), the cube $I$ meets only $O_m(1)$ level-$j$ children. Hence,
\[
\mu(I\cap P)\ll B_j^{-1}.
\]
The first lower limit in \eqref{dimbound} gives $B_j^{-1}\le L_j^s\le r^s$ for all sufficiently large $j$. Thus,
\[
\mu(I)\ll r^s.
\]

By the mass distribution principle, $\hdim K\ge s.$ Letting $s\to s_0$, the desired result follows.
\end{proof}

\section{Preliminaries for Case I: $n=1$}
\label{s:1-case}

From this section on until Section \ref{s:1-proof}, $n=1$ and $m\ge 1$. Let $\psi$ be non-increasing with
\[
\tau_\psi:=\liminf_{q\to\infty}\frac{-\log\psi(q)}{\log q}>\frac1m.
\]
We give some preliminary preparations.

\subsection{Regular scales}

Fix $\ve>0$ so that
\[
0<\varepsilon<\frac16\left(\tau_\psi-\frac1m \right), \ \ m(\tau_\psi-2\varepsilon)>1.
\]
By Lemma~\ref{lem2.1}, there exist arbitrarily large $M$ such that
\begin{equation}\label{eq:1-regular}
\psi(M)\ge M^{-(\tau_\psi+\varepsilon)},\ \ \ \psi(t)\le\psi(M)\left(\frac tM\right)^{-(\tau_\psi-\varepsilon)}\qquad (t\ge M).
\end{equation}
Moreover, for all sufficiently large $t$,
\begin{equation}\label{eq:1-upper}
\psi(t)\le t^{-(\tau_\psi-\varepsilon)},\ \ {\text{and so}}\ \ \sum_{q=1}^\infty\psi(q)^m<\infty.
\end{equation}
Whenever a large scale $M$ is chosen below, it is always chosen from the regular scales of $\psi$.

\subsection{Resonant points and exact cubes}

For $q\ge1$ and $\mathbf p\in\Z^m$, define the \emph{resonant point}
\begin{equation*}
R_{\mathbf p,q}=\frac{\mathbf p+\boldsymbol\theta}{q}.
\end{equation*}
For any $0<c<1$, it follows that
\[
c\psi(q)\le |q\mathbf x-\mathbf p-\boldsymbol\theta|< \psi(q)\ \ \Longleftrightarrow \ \ \frac{c\psi(q)}q\le \left|\mathbf x-R_{\mathbf p,q}\right|<\frac{\psi(q)}q.
\]

Define
\begin{equation*}
\mathcal Q_M=\left\{q\in \N: \frac M2<q\le M,\ q\text{ prime}\right\}
\end{equation*} which is of cardinality comparable with $M/{\log M}$ by prime number theorem.

For $0<\eta<1$, let
\begin{equation}\label{eq:1-child-side}
\ell_M(\eta)=(1-\eta)\frac{\psi(M)}M,
\end{equation}
and, for $q\in\mathcal Q_M$, define the \emph{exact cube}
\begin{equation*}
A_\eta(\mathbf p,q)=R_{\mathbf p,q}+\left[\eta\frac{\psi(q)}q,\,\eta\frac{\psi(q)}q+\ell_M(\eta)\right]^m.
\end{equation*}
It is easy to see that every $\mathbf x\in A_\eta(\mathbf p,q)$ satisfies
\begin{equation*}
\eta\psi(q)\le|q\mathbf x-\mathbf p-\boldsymbol\theta|\le\psi(q).
\end{equation*}

Write a collection of the resonant points as
\[
\mathscr R_M=\left\{R_{\mathbf p,q}: q\in\mathcal Q_M,\ \mathbf p\bmod q\right\}
\]
and a collection of exact cubes as
\[
\mathscr C_M^\eta=\left\{A_\eta(\mathbf p,q): q\in\mathcal Q_M,\ \mathbf p\bmod q\right\}.
\]

\subsection{Dangerous cubes and good cubes}

For each integer $s\ge1$ and $\mathbf r\in\Z^m$, define the \emph{dangerous cube} as
\begin{equation*}
D_\eta(\mathbf r,s)=B\left(\frac{\mathbf r+\boldsymbol\theta}{s}, \eta\frac{\psi(s)}s\right).
\end{equation*}
For a cube $J\subset\T^m$ and $T>1$, we define
\begin{equation*}
\mathfrak B_T^\eta(J)=\sum_{s>T}\sum_{\mathbf r\bmod s}\mathcal L_m\left(J\cap D_\eta(\mathbf r,s)\right),
\end{equation*}which measures the total mass of dangerous cubes of all level $>T$ inside a cube $J$. Since we focus on $\mathbb{T}^m$, so $\mathbf r$ goes over a complete residue class modulo $s$.

\begin{definition}[Good cube] Let $T>1$ and $0<\eta,\delta<1$. A cube $J\subset\T^m$ is called $(T,\eta,\delta)$-\emph{good} if
\begin{equation*}
\mathfrak B_T^\eta(J)\le\delta\mathcal L_m(J).
\end{equation*}
\end{definition}

The following simple estimate shows the existence of good cubes which will provide the initial good parent for latter Cantor
construction: \begin{align*}
\mathfrak B_T^\eta(J)&\le \sum_{s>T}\sum_{\mathbf r\bmod s}\mathcal L_m\left(D_\eta(\mathbf r,s)\right)\asymp \sum_{s>T}s^m\cdot \left(\frac{\psi(s)}s\right)^m\\
&=\sum_{s>T}\psi(s)^m\to 0
\end{align*} as $T\to\infty$. So once $T$ is sufficiently large, we will have
$$\mathfrak B_T^\eta(J)\le\delta\mathcal L_m(J),$$
so $J$ is $(T,\eta,\delta)$-\emph{good}.

\subsection{Auxiliary lemmas}The first two lemmas characterize when two resonant points are close which will be used throughout
 the proof of $n=1$ case.
\begin{lemma}\label{l:1-resonant-1} Let $\delta>0$. Assume that
\[\|R_{\bp,q}-R_{\br,s}\|<\delta\]
for some $\delta>0$. Then, there exists $\mathbf k\in\Z^m$ such that
\begin{equation}\label{eq:1-k-exists}
\left|\mathbf k+(s-q)\boldsymbol\theta\right|<qs\delta,
\end{equation}
and
\begin{equation}\label{eq:1-k-mod}
\mathbf k\equiv s\mathbf p \pmod q,\  \ \  \mathbf k \equiv -q\mathbf r \pmod s, \ \ \ \mathbf k\equiv s\mathbf p-q\mathbf r\pmod{qs}.
\end{equation}
In particular, if $(q,s)=1$, then $\mathbf k$ determines $\mathbf p\bmod q$ and $\mathbf r\bmod s$ uniquely.
\end{lemma}

\begin{proof} The assumption implies that $$
\left|\frac{\mathbf p+\boldsymbol\theta}{q}-\frac{\mathbf r+\boldsymbol\theta}{s}-\boldsymbol{\ell}\right|=\left|\frac{\mathbf p s-\mathbf r q-qs\boldsymbol{\ell}+(s-q)\boldsymbol\theta}{qs}\right|<\delta
$$ for some $\boldsymbol\ell\in \Z^m$.
Put $\mathbf k=s\mathbf p-q\mathbf r-qs\boldsymbol{\ell}$. Then,
\[
\left|\mathbf k+(s-q)\boldsymbol\theta\right|<qs\delta,
\]
and
\[
\mathbf k \equiv s\mathbf p\pmod q, \ \ \ \mathbf k\equiv -q\mathbf r\pmod s, \ \ \ \mathbf k\equiv s\mathbf p-q\mathbf r\pmod{qs}.
\]
If $(q,s)=1$, then $s$ is invertible modulo $q$ and $q$ is invertible modulo $s$. Hence,
\[
\mathbf p \equiv s^{-1}\mathbf k \pmod q, \ \ \ \mathbf r \equiv -q^{-1}\mathbf k\pmod s,
\]
so $\mathbf k$ determines both residue classes uniquely.
\end{proof}
The next lemma deals with a case when $m=1$ and $\theta\in\mathbb Q$.
\begin{lemma}\label{l:1-resonant-2} Suppose that $\theta=b/D$ with $(b,D)=1$. Fix $\delta>0$ and an integer $q>D$. Let
$$\mathcal{S}=\left\{s\in\mathbb N: sq\delta<\frac{1}{ D},\ (q,s)=1\right\}.$$
Then, 
$$\#\Big\{p\ {\rm{ mod }}\ q:\ \Vert R_{p,q}-R_{r,s}\Vert<\delta\text{ for some }s\in\mathcal{S}\Big\}\le 1.$$
\end{lemma}

\begin{proof}
Bear Lemma~\ref{l:1-resonant-1} in mind. By the definition of $\mathcal{S}$, one sees that for any $s\in \mathcal{S}$ the right-hand side of \eqref{eq:1-k-exists} is strictly less than $1/D$, while the left side is a rational number with denominator $D$. This forces
\begin{equation}\label{eq:1-rat-0}
k+(s-q)\frac bD=0, \ \ {\text{so}}\ \ Dk+(s-q)b=0.
\end{equation}
Recall \eqref{eq:1-k-mod} that $k\equiv sp\pmod q$. Thus modulo $q$ on \eqref{eq:1-rat-0} gives
\[
s(Dp+b)\equiv0\pmod q.
\]
Note that $(q,s)=1$, and $(q,D)=1$ since $q$ is prime and $q>D$. Thus, we have
\begin{equation*}
p\equiv-bD^{-1}\pmod q,
\end{equation*}
which proves the lemma.
\end{proof}

The last lemma shows an empty intersection for special exact cubes and dangerous cubes.
\begin{lemma}\label{l:1-resonant-3} Fix $0<\eta<1$ and an arbitrary integer $K\ge 1$. Then, for all sufficiently large $q$, for every
$\bp\text{ mod }q$, and every $\br\text{ mod }s$, we have
$$A_{\eta}(\bp,q)\cap D_{\eta}(\br,kq)=\emptyset$$
for all $k\le K$.
\end{lemma}
\begin{proof}
Assume on the contrary that there exists $\mathbf x$ in $A_\eta(\mathbf p,q)\cap D_{\eta}(\br,kq)$. Write
\begin{equation}\label{eq:1-t-def}
\mathbf t:=q\mathbf x-\mathbf p-\boldsymbol\theta.
\end{equation}
By the definition of the exact cube and dangerous cube, one has
\begin{equation*}
\mathbf t\in\left[\eta\psi(q),\,\eta\psi(q)+q\ell_M(\eta)\right]^m, \ \ \ \ |kq\mathbf x-\mathbf r-\boldsymbol\theta|<\eta\psi(kq).
\end{equation*}

Inserting the expression of $q\mathbf x=\mathbf t+\mathbf p+\boldsymbol\theta$ from (\ref{eq:1-t-def}), it follows that
\[
kq\mathbf x-\mathbf r-\boldsymbol\theta=k\mathbf t+(k-1)\boldsymbol\theta+(k\mathbf p-\mathbf r),
\]
and so
\begin{equation}\label{eq:1-mul-c}
\Vert k\mathbf t+(k-1)\boldsymbol\theta\Vert<\eta\psi(kq)\le \eta\psi(q).
\end{equation}
Since $\mathbf t$ can be arbitrarily small,
we will show that for sufficiently large $q$, (\ref{eq:1-mul-c}) cannot happen for every $k\le K$. Define $$
\epsilon_K=\min\Big\{\|(k-1)\boldsymbol\theta\|: 2\le k\le K, (k-1)\boldsymbol\theta\notin\Z^m\Big\},\ {\text{where}}\ \min{\emptyset}:=+\infty.
$$ Choose $q$ so large that $$
  |K\mathbf{t}|+\eta\psi(q)<\epsilon_K.
  $$
\begin{itemize}
  \item For those $2\le k\le K$ with $(k-1)\boldsymbol\theta\notin\Z^m$, (\ref{eq:1-mul-c}) cannot happen, since
   $$\epsilon_K\le \|(k-1)\boldsymbol\theta\|\le  |k\mathbf t|+\eta\psi(kq)<\epsilon_K.$$

\item For those $2\le k\le K$ with 	$(k-1)\boldsymbol\theta\in\Z^m$, \eqref{eq:1-mul-c} reduces to
\begin{equation}\label{eq:1-contradicts}
\|k\mathbf t\|<\eta\psi(kq).
\end{equation}
For sufficiently large $q$, since $\mathbf t$ is so small that for every $k\le K$
\[
\|k\mathbf t\|=|k\mathbf t|\ge |\mathbf{t}|>\eta\psi(q).
\]
This contradicts \eqref{eq:1-contradicts}, since $\psi$ is non-increasing. \qedhere
\end{itemize}
\end{proof}
\subsection{Main proposition}
The next proposition describes the separation properties of the resonant points, which is vital for the dimension of $\Exact_{m,1}(\psi,\boldsymbol\theta)$. Starting from a good exact cube $J$, we will show that there are sufficiently many well separated good exact cubes inside $J$ which can avoid all dangerous cubes in previous levels.

Choose $a>0$ small and $\rho>0$ close to $1+1/m$ such that
\begin{equation}\label{eq:1-rho}
a(m+1)<1, \ \ \ 1+\frac1m<\rho<\min\left\{1+\tau_\psi-\varepsilon,\,\frac1m+\frac{m(1+\tau_\psi+2\varepsilon)}{m+1}\right\}.
\end{equation}

\begin{proposition}\label{p:1-tran} Let $T>1,  0<\eta_0<\eta<1,$ $0<\delta<1$ and let $J\subset\T^m$ be a $(T,\eta_0,\delta)$-good cube of side
length $L$. Then, for every sufficiently large regular scale $M$ with $M^a>T$,
one can retain at least
\begin{equation*}
c_1L^m \frac{M^{m+1}}{\log M}
\end{equation*}
exact cubes in $\mathscr{C}_M^\eta$, where $c_1=c_1(m,\sigma,\delta)>0$, satisfying
\begin{enumerate}
\item [(a)] they are disjoint from every $D_{\eta_0}(\mathbf r,s)$ with $T<s\le M;$

\item [(b)] they are pairwise $\gg M^{-\rho}$-separated;

\item [(c)] all of them are $(M,\eta,\delta)$-good.
\end{enumerate}
\end{proposition}

Sections \ref{s:1-low}--\ref{s:1-large} are devoted to the proof of Proposition~\ref{p:1-tran}. After proving it, we shall apply it recursively to construct a Cantor subset of $\Exact_{m,1}(\psi,\boldsymbol\theta)$, and then estimate its Hausdorff dimension.

\section{Proof of Proposition \ref{p:1-tran}($a$)}\label{s:1-low}

We remove those exact cubes in $\mathscr{C}_M^\eta$ which can intersect dangerous cubes coming from the small
denominators $T<s<M^a$ and intermediate denominators $M^a\le s\le M$ separately.

\subsection{Small denominators $T<s<M^a$}

We first remove those exact cube in $\mathscr{C}_M^\eta$ which can intersect dangerous cubes of small
denominators $T<s<M^a$.
The goodness of the parent cube ensures that the dangerous cubes take only a small proportion inside the parent cube, so not too much exact cubes in $\mathscr{C}_M^\eta$ need to be removed.

\begin{proposition}\label{p:1-low} Fix $0<\eta_0<\eta<1$, let $J$ be a $(T,\eta_0,\delta)$-good cube of side length $L$.
There exists $c_1=c_1(m,\delta)>0$ such that, for every sufficiently large $M$ with $M^{a(m+1)-1}=o(L^m)$, at least
\begin{equation*}
\frac{c_1L^mM^{m+1}}{\log M}
\end{equation*}
exact cubes $A_\eta(\mathbf p,q)$ with $q\in\mathcal Q_M$ are contained in $J$ and are disjoint from every dangerous cube $D_{\eta_0}(\mathbf r,s)$ with $T<s<M^a$.
\end{proposition}

\begin{proof}
Let $B_1,\ldots,B_N$ be all the dangerous cubes $D_{\eta_0}(\mathbf r,s)$ with $T<s<M^a$.  Since, for each fixed $s$, there are $s^m$ residue classes $\mathbf r\bmod s$, we have
\begin{equation*}
N\le\sum_{s<M^a}s^m=O_m(M^{a(m+1)}).
\end{equation*}
Since $B_1,\ldots,B_N$ contribute only part to the sum defining $\mathfrak B_T^{\eta_0}(J)$, the $(T,\eta_0,\delta)$-goodness of $J$ gives
\begin{align}\label{eq:11-removed}
\sum_{i=1}^N\mathcal L_m(J\cap B_i)&=\sum_{T<s<M^a}\sum_{\mathbf r\bmod s}\mathcal L_m\left(J\cap D_{\eta_0}(\mathbf r,s)\right)
\le\mathfrak B_T^{\eta_0}(J)\le\delta L^m.	
	\end{align}

We count the number of the resonant points $R_{\mathbf p, q}\in \mathscr R_M\cap J$ but far from the boundary of $J$ and the dangerous cubes ensuring that
\begin{equation}\label{eq:11-surv}
    A_\eta(\mathbf p,q)\subset J \qaq  A_\eta(\mathbf p,q)\cap D_{\eta_0}(\mathbf r,s)=\emptyset
\end{equation}
for every $T<s<M^a.$

Denote $\widetilde B_i$ the $M^{-1}$-neighborhood of $B_i$ and remove a strip of width $M^{-1}$ along the boundary of $J$. So define
\[
J_{-1/M}:=\left\{\mathbf x\in J: \operatorname{dist}(\mathbf x,J^c)\ge\frac1M\right\}, \ {\text{and}}\
F:=J_{-1/M}\setminus\bigcup_{i=1}^N\widetilde B_i.
\]

Since $q\asymp M$ for $q\in\mathcal Q_M$ and the diameter of $A_\eta(\mathbf p,q)$ is $o(M^{-1})$, then the inclusion in \eqref{eq:11-surv} is guaranteed once $R_{\mathbf p,q}\in F$.
So it suffices to count the resonant points contained in $F$.

$\bullet$ {\em The total number of resonant points $R_{\mathbf p,q}$ in $J_{-1/M}$ for a fixed $q\in\mathcal Q_M$.}

For fixed $q\in\mathcal Q_M$, the resonant points $\{R_{\mathbf p,q}: \mathbf p \bmod q\}$ are $q^{-1}$ separated. A volume estimate gives
\begin{align}\label{eq:11-grid}
\#\left\{\mathbf p\bmod q: R_{\mathbf p,q}\in J_{-1/M}\right\}&=q^m\mathcal L_m(J)+O\bigl((qL+1)^{m-1}\bigr)\nonumber\\
&=q^m \mathcal L_m(J)+O(M^{m-1}).
\end{align}

$\bullet$ {\em The number of resonant points $R_{\mathbf p,q}$ in $J_{-1/M}$ for a fixed $q\in\mathcal Q_M$ to be deleted.}

For $1\le i\le N$, let
\[
C_i:=J_{-1/M}\cap\widetilde B_i.
\]
Each $C_i$ is an axis-aligned box, and the same volume estimate yields
\begin{align*}
\#\left\{\mathbf p\bmod q: R_{\mathbf p,q}\in C_i\right\}&\le q^m\mathcal L_m(C_i)+O\bigl((q|\widetilde B_i|+1)^{m-1}\bigr)\\
&\le  q^m\mathcal L_m(J\cap B_i)+O(M^{m-1}).
\end{align*}
Thus the number of resonant points  to be deleted for a fixed $q\in\mathcal Q_M$ is bounded by
\begin{align}\label{eq:11-delete}
\le \sum_{i=1}^N\Big(q^m\mathcal L_m(J\cap B_i)+O(M^{m-1})\Big)=q^m\sum_{i=1}^NL_m(J\cap B_i)+O(M^{a(m+1)+m-1}).
\end{align}

$\bullet$ {\em The number of retained resonant points.}	Combining (\ref{eq:11-grid}) and (\ref{eq:11-delete}), we conclude that the number of resonant points $R_{\mathbf p,q}\in F$ with $q\in\mathcal Q_M$
\begin{align}
&\ge \sum_{q\in\mathcal Q_M}\left(q^m \mathcal L_m(J)+O(M^{m-1})-q^m\sum_{i=1}^N\mathcal{L}_m(J\cap B_i)+O(M^{a(m+1)+m-1})\right)\nonumber\\
&\ge \left((1-\delta)\mathcal L_m(J)\cdot\sum_{q\in\mathcal Q_M}q^m\right)+\sum_{q\in \mathcal Q_M}O\left(M^{a(m+1)+m-1}\right)\nonumber\\
&\asymp \frac{M^{m+1}}{\log M}\cdot L^m+O\left(\frac{M^{a(m+1)+m}}{\log M}\right)\gg_{\delta}L^m\frac{M^{m+1}}{\log M},\nonumber
\end{align}
where the second inequality follows from the goodness of $J$ (see (\ref{eq:11-removed})), and the last relation is because
$$a(m+1)<1,\ {\text{and then}}\ M^{1-a(m+1)}=o(L^m)$$
since $M$ can be sufficiently large.
\end{proof}

\subsection{Intermediate denominators $M^a\le s\le M$}

\begin{proposition}\label{p:1-int} Denote $N_<(M)$ the number of exact cubes $A_\eta(\mathbf p,q)$ in $\mathscr C_M^\eta$ which intersect at least one dangerous cube $D_{\eta_0}(\mathbf r,s)$ of denominators $M^a\le s\le M.$ Then, as $M\to\infty$,
\begin{equation*}
N_<(M) = o\left( \frac{M^{m+1}}{\log M} \right).
\end{equation*}
In particular, for any fixed $L>0$, by choosing $M$ sufficiently large,
\[ N_<(M) = o\left( L^m\frac{M^{m+1}}{\log M} \right). \]
\end{proposition}

\begin{proof}
We are aimed at counting the number of the set: \begin{align*}
 {\mathcal{B}ad}&:= \Big\{A_\eta(\mathbf p,q)\in \mathscr C_M^\eta: A_\eta(\mathbf p,q)\cap D_{\eta_0}(\mathbf r,s)\ne\emptyset, \ {\text{for some}}\ M^a\le s\le M, {\text{and}}\ \mathbf r \bmod s\Big\}\\
&\subset  \bigcup_{q\in \mathcal Q_M}\bigcup_{M^a\le s<M}\Big\{A_\eta(\mathbf p,q)\in \mathscr C_M^\eta: A_\eta(\mathbf p,q)
\cap D_{\eta_0}(\mathbf r,s)\ne\emptyset, \ {\text{for some}}\ \mathbf r \bmod s\Big\}\\
&:=\bigcup_{q\in \mathcal Q_M}\bigcup_{M^a\le s<M}{\mathcal{B}ad}(q,s).
\end{align*}

 For any $q\in\mathcal Q_M$, the exact cube $A_\eta(\mathbf p,q)$ is disjoint from all dangerous cubes $D_{\eta_0}(\mathbf r,s)$ with $s=q$:
\begin{itemize}
  \item if $\mathbf r=\mathbf p$, this is clear by the definition of $A_\eta(\mathbf p,q)$;
  \item if $\mathbf r\ne \mathbf p$, the resonant points $R_{\mathbf r, q}$ and $R_{\mathbf p, q}$ are at least $1/q$ separated.
\end{itemize}

So we need only care about those ${\mathcal{B}ad}(q,s)$ for $s\ne q$. Since $q$ is prime and $q>M/2$, we have $s\le M<2q$. Thus, if $s\ne q$, then $q\nmid s$, and hence
$
(q,s)=1.
$

Fix $q$ and $s$ with $(q,s)=1$. Assume that
\[
A_\eta(\mathbf p,q) \cap D_{\eta_0}(\mathbf r,s) \ne\emptyset.
\]
Then,
\begin{equation}\label{eq:12-reso-close}\|R_{\bp,q}-R_{\br,s}\|<\eta_0\frac{\psi(s)}{s}+\frac{\psi(q)}{q}<\frac{\psi(s)}{s}+\frac{\psi(q)}{q}.\end{equation}
By Lemma~\ref{l:1-resonant-1}, there exists $\mathbf k\in\Z^m$ satisfying
\begin{equation}\label{eq:12-k}
\left|\mathbf k+(s-q)\boldsymbol\theta\right|< q\psi(s)+s\psi(q),
\end{equation}
and
\begin{equation*}
\mathbf k \equiv s\mathbf p \pmod q, \ \ \  \mathbf k \equiv -q\mathbf r \pmod s.
\end{equation*}
Since $(q,s)=1$, every such $\mathbf k$ determines uniquely $\mathbf p\bmod q$ and $\mathbf r\bmod s$. Then we are led to count the number of possible $\mathbf k$ satisfying (\ref{eq:12-k}).

We note that if $s$ is small such that $$
q\psi(s)+s\psi(q)\gg 1,
$$ there will be multiple choices of $k$ by (\ref{eq:12-k}), otherwise there is at most one. So the argument will be divided into cases according to $s$ small or large (not necessary for $m\ge 2$).

\medskip

\noindent{\bf Case 1: $m\ge 2$.} For any fixed $q$ and $s$, the inequality \eqref{eq:12-k} restricts $\mathbf k$ to a cube of side length less than $ q\psi(s)+s\psi(q)$. Hence,
\begin{equation*}
\#\mathcal{B}ad(q,s)\ll (q\psi(s)+s\psi(q)+1)^m\ll q^m\psi(s)^m+s^m\psi(q)^m+1.
\end{equation*}
As a result,\begin{align*}
\#\mathcal{B}ad &\ll \sum_{q\in\mathcal Q_M}\sum_{M^a\le s<M}\left(q^m\psi(s)^m+s^m\psi(q)^m+1\right)\nonumber\\
&\ll_m\frac{M^{m+1}}{\log M}\sum_{s\ge M^a}\psi(s)^m+M^{m+1}\sum_{q\in\mathcal Q_M}\psi(q)^m+\frac{M^2}{\log M}\nonumber\\
&=o\left(\frac{M^{m+1}}{\log M}\right),
\end{align*}
where the last assertion is because of the facts
\[
\sum_{s\ge M^a}\psi(s)^m \rightarrow0, \ \ \ \ \sum_{q\in\mathcal Q_M}\psi(q)^m\le \sum_{q\in\mathcal Q_M}q^{-m(\tau_\psi-\varepsilon)}=O(M^{-2\varepsilon}), \ \ {\text{and}}\ m\ge 2.
\]

\noindent{\bf Case 2: $m=1$.} The difference between $m=1$ and $m\ge2$ is that the error term $M^2/\log M$ cannot be controlled by $M^{m+1}/\log M$ when $m=1$.

We first restrict our attention to $\mathcal{B}ad(q,s)$ with $M^a\le s\le \frac{M}{\log M}$.
By the same argument as in case 1, we have
$$\#\mathcal{B}ad(q,s)\le  q\psi(s)+s\psi(q)+1=O(M\psi(s)+1). $$
Thus, the number of exact cubes in $\mathcal{B}ad$ contributed by those $(q, s)$ is bounded by
\begin{align*}
\le &\ \sum_{q\in \mathcal Q}\sum_{M^a\le s\le \frac{M}{\log M}}\#\mathcal{B}ad(q,s)
\ll \frac M{\log M}\sum_{M^a\le s\le \frac{M}{\log M}}\left(M\psi(s)+1\right)\nonumber\\
\ll &\ \frac{M^2}{\log M}\sum_{s\ge M^a}\psi(s)+\frac{M^2}{(\log M)^2}=o\left(\frac{M^2}{\log M}\right).
\end{align*}

It remains to consider $\mathcal{B}ad(q,s)$ for $\frac{M}{\log M}<s<M$.
\begin{itemize}
\item If $\theta\in\mathbb{Q}.$ Write $\theta=b/D$ in the lowest term. Recall \eqref{eq:12-reso-close} that
\begin{equation*}
\|R_{\bp,q}-R_{\br,s}\|<\frac{\psi(s)}{s}+\frac{\psi(q)}{q}:=\delta.
\end{equation*}
For any sufficiently large $M$, $q\in \mathcal Q_M$ and ${M}/{\log M}<s<M$, we have
\begin{equation}\label{eq:12-qsd}
qs\delta=q\psi(s)+s\psi(q)\le M \bigg(\frac{M}{\log M}\bigg)^{-\tau_{\psi}+\varepsilon}
+M \bigg(\frac{M}{2}\bigg)^{-\tau_{\psi}+\varepsilon}<\frac{1}{D}.
\end{equation}
Therefore, Lemma~\ref{l:1-resonant-2} implies that for each $q\in\mathcal Q_M$,
$$
\#\left(\bigcup_{\frac{M}{\log M}<s<M}\mathcal{B}ad(q,s)\right)\le 1,
$$ and so
\begin{align*}
\#\left(\bigcup_{q\in \mathcal Q}\bigcup_{\frac{M}{\log M}<s<M}\mathcal{B}ad(q,s)\right)
&\le \sum_{q\in \mathcal Q}1=O\left(\frac M{\log M}\right)=o\left(\frac{M^2}{\log M}\right).
\end{align*}
\item If $\theta\notin\mathbb{Q}$. By the estimation given in (\ref{eq:12-qsd}), the inequality (\ref{eq:12-k}) reduces to
$$\left|k+(s-q)\theta\right|\le 2 M\bigg(\frac{M}{\log M}\bigg)^{-\tau_\psi+\varepsilon}<1/2.$$
This implies that if \[
A_\eta(p,q)\cap D_{\eta_0}(r,s)\ne\emptyset,
\] for some $M/\log M<s<M$, there is a unique choice of $k$ (so uniquely $p\mod q$ and $r\mod s$) and also $s$ satisfies
\[
\|(q-s)\theta\|\le2 M\bigg(\frac{M}{\log M}\bigg)^{-\tau_\psi+\varepsilon}.
\]
Applying Lemma~\ref{lem2.3} (Denjoy--Koksma inequality) to $h=q-s$ and $K=M$ gives
\[
\#\left\{\frac{M}{\log M}<s<M: \|(q-s)\theta\|\le 2 M \bigg(\frac{M}{\log M}\bigg)^{-\tau_\psi+\varepsilon}\right\}\ll 2 M^2\bigg(\frac{M}{\log M}\bigg)^{-\tau_\psi+\varepsilon}+\frac{M}{Q_\theta(M)}.
\]
Hence, the total number of $\mathcal{B}ad(q,s)$ in this range is
\begin{equation*}
\#\left(\bigcup_{q\in \mathcal{Q}}\bigcup_{\frac{M}{\log M}\le s\le M}\mathcal{B}ad(q,s)\right)\ll\frac M{\log M}\left(M^2\bigg(\frac{M}{\log M}\bigg)^{-\tau_\psi+\varepsilon}+\frac{M}{Q_\theta(M)}\right)=o\left(\frac{M^2}{\log M}\right),
\end{equation*}
since $\tau_{\psi}-\varepsilon>1$ and $Q_\theta(M)\to\infty.$
\end{itemize}

\vskip -15pt\end{proof}

\section{Proof of Proposition \ref{p:1-tran}($b$)}

We show that, after removing only a negligible number of resonant points from $\mathscr R_M$, the remaining resonant points are $M^{-\rho}$-separated.  This will in turn give the required separation between the corresponding exact cubes.

Recall from \eqref{eq:1-rho} that
\begin{equation*}
1+\frac1m<\rho<\min\left\{1+\tau_\psi-\varepsilon,\,\frac1m+\frac{m(1+\tau_\psi+2\varepsilon)}{m+1}\right\}.
\end{equation*}

Denote $E_M$ the collection of resonant points in $\mathscr R_M$ which is $M^{-\rho}$ close to some other resonant points in $\mathscr R_M$, i.e.
$$
E_M=\Big\{R_{\mathbf p,q}\in \mathscr R_M: \exists \ R_{\mathbf p',q'}\in \mathscr R_M \text{ such that } R_{\mathbf p,q}\ne R_{\mathbf p',q'} \text{ and } \|R_{\mathbf p,q}-R_{\mathbf p',q'}\|<M^{-\rho}\Big\}.
$$
\begin{proposition} Let $E_M$ be defined as above.
Then,
\begin{equation}\label{eq:13-num}
\# E_M=o\left(\frac{M^{m+1}}{\log M}\right).
\end{equation}
Consequently, for any fixed $L>0$, if $M$ is sufficiently large, by removing at most
	\[
	o\left(L^m\frac{M^{m+1}}{\log M}\right)
	\]
resonant points in $\mathscr R_M$, the remaining resonant points are $M^{-\rho}$-separated, and then the corresponding exact cubes are pairwise $\gg M^{-\rho}$ separated.
\end{proposition}
\begin{proof}
We divide the proof into two cases.
\begin{itemize}
\item Resonant points with the same denominator $q$. If $R_{\mathbf p,q}\neq R_{\mathbf p',q},$ then $\mathbf p\not\equiv\mathbf p'\pmod q$, and therefore
\[
\|R_{\mathbf p,q}-R_{\mathbf p',q}\|\ge\frac1q\gg M^{-1}>M^{-\rho},
\]
i.e. they cannot be $M^{-\rho}$-close to each other.

\item Resonant points with different denominators. Assume that $q,q'\in\mathcal Q_M$, with $q\neq q'$ and
\[
\|R_{\mathbf p,q}-R_{\mathbf p',q'}\|<M^{-\rho}.
\]
By Lemma~\ref{l:1-resonant-1}, there exists $\bk\in\mathbb{Z}^m$ such that
\begin{equation}\label{eq:13-close}
\left|\mathbf k+(q'-q)\boldsymbol\theta\right|\ll M^{2-\rho},
\end{equation}
since $q,q'\asymp M$. Moreover,
\begin{equation*}
\mathbf k\equiv q'\mathbf p\pmod q,\ \ \ \  \mathbf k \equiv -q\mathbf p'\pmod{q'}.
\end{equation*}
Since the denominators $q$ and $q'$ are distinct primes, we have $(q,q^{\prime})=1$.
Thus, every vector $\mathbf k$ determines uniquely both residue classes
\[
\mathbf p\bmod q\ \ \ \text{and}\ \ \  \mathbf p'\bmod q'.
\]
By \eqref{eq:13-close}, the possible vectors $\mathbf k$ lie in an $m$-dimensional cube of side length $O(M^{2-\rho})$.  Hence their number is
\[
O_m\left((M^{2-\rho}+1)^m\right)\ll_m M^{m(2-\rho)}+1.
\]
Summing over all pairs $(q, q')$ of distinct denominators gives
\[
\# E_M\ll_m \frac{M^2}{(\log M)^2}\left(M^{m(2-\rho)}+1\right).
\]
Since $\rho>1+1/m$, it is easy to see that, as $M\rightarrow\infty$
\begin{align*}
\frac{M^{2+m(2-\rho)}/(\log M)^2}{M^{m+1}/\log M}=\frac{M^{1+m-m\rho}}{\log M}\rightarrow0, \ \  \ \frac{M^2/(\log M)^2}{M^{m+1}/\log M}=\frac{M^{1-m}}{\log M}\rightarrow0.
\end{align*}
This proves \eqref{eq:13-num}.
\end{itemize}

Finally, for $q\in\mathcal Q_M$, the distance from any point of $A_\eta(\mathbf p,q)$ to its resonant point is at most
\[
\eta\frac{\psi(q)}q+\ell_M(\eta)\ll\frac{\psi(q)}{q}\ll M^{-(1+\tau_\psi-\varepsilon)}=o(M^{-\rho}).
\]
Since their resonant points are $M^{-\rho}$-separated, so are the corresponding exact cubes.
\end{proof}

\section{Proof of Proposition \ref{p:1-tran}($c$)}\label{s:1-large}

We turn now to the final selection step: all but a negligible number of exact cubes in $\mathscr C_M^\eta$ are $(M,\eta,\delta)$-good.

Note that the total measure of cubes in $\mathscr C_M^\eta$ is
\begin{equation}\label{eq:14-mass}
S_M(\eta):=\sum_{A_\eta(\mathbf p,q)\in\mathscr C_M^\eta}\mathcal L_m(A_\eta(\mathbf p,q))\asymp (1-\eta)^m\frac M{\log M}\psi(M)^m.
\end{equation}
We shall prove that the total measure $\mathcal I_M(\eta)$ of the intersections between the exact cubes in $\mathscr{C}_M^\eta$ and the dangerous cubes with denominators $s>M$ is $o(S_M(\eta))$ as $M\to\infty$.  This enables us to conclude that only a negligible number of exact cubes in $\mathscr C_M^\eta$ fail to be $(M,\eta,\delta)$-good.

We divide $\mathcal I_M(\eta)$ into two parts according to whether $q$ divides $s$ or not, so define
\begin{align}\label{eq:14-nm}
\mathcal I_M^{(1)}(\eta)&:=\sum_{A_\eta(\mathbf p,q)\in\mathscr C_{M}^\eta}\sum_{\substack{s>M\\ q\mid s}}\sum_{\mathbf r\bmod s}\mathcal L_m\left(A_\eta(\mathbf p,q)\cap D_\eta(\mathbf r,s)\right),\notag\\
\mathcal I_M^{(2)}(\eta)&:=\sum_{A_\eta(\mathbf p,q)\in	\mathscr C_{M}^\eta}\sum_{\substack{s>M\\ q\nmid s}}\sum_{\mathbf r\bmod s}\mathcal L_m\left(A_\eta(\mathbf p,q)\cap D_\eta(\mathbf r,s).\right).
\end{align}
In the next two lemmas, we show separately that $$
\mathcal I_M^{(1)}(\eta)=o(S_M(\eta))\quad\text{and}\quad \mathcal I_M^{(2)}(\eta)=o(S_M(\eta)).
$$

\begin{proposition} With $\mathcal I_M^{(1)}(\eta)$ defined above, we have
\begin{equation*}
\mathcal I_M^{(1)}(\eta)=o(S_M(\eta))\text{ as }M\to\infty.	
\end{equation*}
\end{proposition}

\begin{proof}Recall Lemma~\ref{l:1-resonant-3} that for every fixed $K$, for all sufficiently large $M$, one has
\begin{equation}\label{eq:14-m-empty}
A_\eta(\mathbf p,q)\cap D_\eta(\mathbf r,kq)=\emptyset, \ \ {\text{for all}}\ \ A_\eta(\mathbf p,q)\in\mathscr C_M^\eta, \ k\le K, \ \mathbf r \bmod s.
\end{equation}

Since $m(\tau_\psi-\varepsilon)>1$, for any $\delta>0$, we choose $K$ such that $$
 \sum_{k>K}k^{-m(\tau_\psi-\varepsilon)}+k^{-m(1+\tau_\psi-\varepsilon)}<\delta.
$$ Then choose $M$ such that (\ref{eq:14-m-empty}) is valid. So the contribution in $\mathcal I_M^{(1)}(\eta)$ only happens for those $s=kq$ with $k>K$.

Fix an exact cube $A_\eta(\mathbf p,q)\in\mathscr C_M^\eta$ which has side length $\ell_M(\eta).$ For any $k> K$, the centers
$$
\frac{\mathbf r+\boldsymbol\theta}{kq}, \ \  \mathbf r\bmod kq
$$ of the dangerous cubes
$\{D_\eta(\mathbf r,kq): \mathbf r\bmod kq\}$,
form a grid with spacing $1/(kq)$ in each coordinate. Hence, $A_\eta(\mathbf p,q)$ can intersect at most
\[
O\bigl((kq\ell_M(\eta)+1)^m\bigr)
\]
such dangerous cubes.
Consequently, we have
\begin{align*}
&\sum_{\mathbf r\bmod kq}\mathcal L_m\left(A_\eta(\mathbf p,q)\cap D_\eta(\mathbf r,kq)\right)
\ll\bigl(kq\ell_M(\eta)+1\bigr)^m\left(\frac{\psi(kq)}{kq}\right)^m.
\end{align*}
Dividing by $\mathcal L_m(A_\eta(\mathbf p,q))=\ell_M(\eta)^m,$ we obtain
\begin{align}
\frac{\displaystyle\sum_{\mathbf r\bmod kq}\mathcal L_m\left(A_\eta(\mathbf p,q)\cap D_\eta(\mathbf r,kq)\right)}{\mathcal L_m(A_\eta(\mathbf p,q))}
&\ll\left(\psi(kq)+\frac{\psi(kq)}{kq\ell_M(\eta)}\right)^m
\ll\psi(kq)^m+\left(\frac{\psi(kq)}{k\psi(M)}\right)^m,\nonumber
\end{align}
where in the last step we used
\[
q\ell_M(\eta)\stackrel{\eqref{eq:1-child-side}}{\asymp}_{\eta}\psi(M).
\]
Since $kq>M$, one has
\[
\psi(kq)\stackrel{\eqref{eq:1-regular}}{\ll}\psi(M)k^{-(\tau_\psi-\varepsilon)}.
\]
Therefore,
\begin{align*}
\frac{\displaystyle\sum_{\mathbf r\bmod kq}\mathcal L_m\left(A_\eta(\mathbf p,q)\cap D_\eta(\mathbf r,kq)\right)}
{\mathcal L_m(A_\eta(\mathbf p,q))}\ll k^{-m(\tau_\psi-\varepsilon)}+k^{-m(1+\tau_\psi-\varepsilon)}.
\end{align*}
Summing over $k> K$ and over $A_\eta(\mathbf p,q)\in\mathscr C_M^\eta$, we complete the proof of the lemma.
\end{proof}

\begin{proposition}
With $\mathcal I_M^{(2)}(\eta)$ defined in \eqref{eq:14-nm}, we have
\begin{equation*}
\mathcal I_M^{(2)}(\eta)=o(S_M(\eta))\text{ as }M\to\infty.
\end{equation*}
\end{proposition}
\begin{proof}
Suppose that
$s>M,q\nmid s,{\text{ and }}$
$$A_\eta(\mathbf p,q)\cap D_\eta(\mathbf r,s)\neq\emptyset.$$
Since $q$ is prime and $q\nmid s$, we have
$(q,s)=1.$
As in the proof of Proposition~\ref{p:1-int}, we treat the following two cases.

	\noindent{\bf Case 1: $m>1$.} By Lemma \ref{l:1-resonant-1}, the intersection implies
\begin{equation*}
|\mathbf k+(s-q)\boldsymbol\theta|\le q\eta\psi(s)+s\psi(q)\le 2s\psi(q).
\end{equation*}
Hence, for fixed $(q,s)$, the number of possible integer vectors $\mathbf k$ is at most
\[
\ll\left(2s\psi(q)+1\right)^m\ll s^m\psi(q)^m+1.
\]
By Lemma \ref{l:1-resonant-1} and $(q,s)=1$, each such $\mathbf k$ determines uniquely both $\mathbf p\bmod q$ and $\mathbf r\bmod s$. Furthermore, each corresponding intersection gives volume at most that of a dangerous cube, so
\begin{align}
\mathcal I_M^{(2)}(\eta)&\ll_m\sum_{q\in\mathcal Q_M}\sum_{\substack{s>M\\ q\nmid s}}\left(s^m\psi(q)^m+1\right)\left(\frac{\psi(s)}s\right)^m.\nonumber
\end{align}

For $s>M$, by \eqref{eq:1-regular},
\[
\psi(s)\ll\psi(M)\left(\frac{s}{M}\right)^{-(\tau_\psi-\varepsilon)}.
\]
Hence, for each fixed $q\in\mathcal Q_M$,
\begin{align*}
\sum_{s>M}\left(s^m\psi(q)^m+1\right)\left(\frac{\psi(s)}s\right)^m=&\ \psi(q)^m\sum_{s>M}\psi(s)^m+\sum_{s>M}\frac{\psi(s)^m}{s^m}\nonumber\\
\ll &\ \psi(q)^m\psi(M)^m\sum_{s>M}\left(\frac{s}{M}\right)^{-(\tau_\psi-\varepsilon)}+\psi(M)^m\sum_{s>M}s^{-m}\nonumber\\
\ll &\ M\psi(q)^m\psi(M)^m+M^{1-m}\psi(M)^m.
\end{align*}
Moreover, since $q\asymp M$, by
\eqref{eq:1-upper},
\[
\psi(q)\ll M^{-(\tau_\psi-\varepsilon)},
\]
one has
\[
\sum_{s>M}\left(s^m\psi(q)^m+1\right)\left(\frac{\psi(s)}s\right)^m \ll M^{1-m(\tau_{\psi}-\varepsilon)}\psi(M)^{m}+M^{1-m}\psi(M)^m.
\]
Therefore, by $\tau_{\psi}-\varepsilon>1$ and $m>1$,
\[
\mathcal I_M^{(2)}(\eta)\ll\frac M{\log M}\left(M^{1-m(\tau_{\psi}-\varepsilon)}+M^{1-m}\right)\psi(M)^m=o\left(\frac M{\log M}\psi(M)^m\right)=o(S_M(\eta)).
\]

\noindent{\bf Case 2: $m=1$.} For fixed $s>M^{\tau_{\psi}-2\varepsilon}$ and $q\in\mathcal Q_M$, the number of possible integers $k$ is at most
$$
O(s\psi(q)+1).
$$
Since $(q,s)=1$, each such $k$ uniquely determines $p\bmod q$ and $r\bmod s$. Each corresponding intersection has length at most $O(\psi(s)/s)$. Hence, for each fixed $q$,
\begin{align}
\sum_{s>M^{\tau_{\psi}-2\varepsilon}}
O(s\psi(q)+1)\cdot\frac{\psi(s)}{s}
&\ll
\psi(q)\sum_{s>M^{\tau_{\psi}-2\varepsilon}}\psi(s)
+
\sum_{s>M^{\tau_{\psi}-2\varepsilon}}\frac{\psi(s)}{s}.\label{eq:14-s>}
\end{align}

Since $M$ is a regular scale and $M/2\le q<M$, one has
$$\psi(q)\le q^{-\tau_{\psi}+\varepsilon}\ll M^{2\varepsilon}\psi(M).$$
Moreover, $M^{\tau_{\psi}-2\varepsilon}\ge M$. Thus, for $s>M^{\tau_{\psi}-2\varepsilon}$,
\[
\psi(s)\le s^{-\tau_{\psi}+\varepsilon}\qaq\psi(s)\le \psi(M)\left(\frac Ms\right)^{\tau_\psi-\varepsilon}.
	\]
As a consequence,
\begin{align*}
\eqref{eq:14-s>}&\ll M^{2\varepsilon}\psi(M)\sum_{s>M^{\tau_{\psi}-2\varepsilon}}s^{-\tau_{\psi}+\varepsilon}+\psi(M)\sum_{s>M^{\tau_{\psi}-2\varepsilon}}
\left(\frac Ms\right)^{\tau_\psi-\varepsilon}\cdot \frac{1}{s}\\
& \ll M^{2\varepsilon}\psi(M)M^{(\tau_{\psi}-2\varepsilon)(1-\tau_{\psi}+\varepsilon)}+
\psi(M)M^{(\tau_{\psi}-\varepsilon)(1-\tau_{\psi}+2\varepsilon)}\\
&=o(\psi(M)).
\end{align*}
Summing over $q\in\mathcal Q_M$,
we obtain the contribution to $\mathcal I_M^{(2)}(\eta)$ for $s> M^{\tau_{\psi}-2\varepsilon}$ is
\[
o\left(\frac{M}{\log M}\psi(M)\right)=o(S_M(\eta)).
\]

It remains to consider $M<s\le M^{\tau_{\psi}-2\varepsilon}$.

\begin{itemize}
\item If $\theta\in\mathbb{Q}$. Write $\theta=b/D$ in its lowest term. By Lemma~\ref{l:1-resonant-1}, there exists $k\in\Z$ such that
\begin{equation}\label{eq:14-k-rat}
\left|k+(s-q)\frac bD\right|\le 2s\psi(q).
\end{equation}
The choice of $s$ forces the right-hand side of	\eqref{eq:14-k-rat} is less than $1/2D$ and so the left-hand side equal 0. Hence,
$$\left\|\frac{ p+\theta}{q}-\frac{r+\theta}{s}\right\|=\frac{1}{qs}\Big|k+(s-q)\theta\Big|=0.$$
This means that the resonant points $R_{p,q}$ and $R_{r,s}$ coincide on $\T$. Since $s>M\ge q,$ by the monotonicity of $\psi$, it follows that
\[
\eta\frac{\psi(s)}s<\eta\frac{\psi(q)}q.
\] So it is impossible that $A_\eta(\mathbf p,q)\cap D_\eta(\mathbf r,s)\neq\emptyset.$
Therefore, no exact interval in $\mathscr C_M^\eta$ can intersect a dangerous interval with $M<s\le M^{\tau_{\psi}-2\varepsilon}.$

\item If $\theta\notin\mathbb{Q}$.
Fix $q\in\mathcal Q_M$. By Lemma~\ref{l:1-resonant-1} and $M/2< q\le M$, there exists $k\in\Z$ such that
\begin{equation}\label{eq:14-k-irr}
|k+(s-q)\theta|\le 2s\psi(q)\le 2^{1+\tau_{\psi}-\varepsilon}sM^{-\tau_{\psi}+\varepsilon}.
\end{equation}
Note that $(q,s)=1$. So, for any pair $(q, s)$, the integer $k$ determines $p\bmod q$ and $r\bmod s$ uniquely. The right side of \eqref{eq:14-k-irr} is smaller than $1/2$, so $k$ has only one choice and $s$ satisfies $$
\|(s-q)\theta\|\le 2^{1+\tau_{\psi}-\varepsilon}sM^{-\tau_{\psi}+\varepsilon}.
$$
Divide the range of $s$ into consecutive blocks of length at most $M$,
$S\le s<S+M.$
For each block, we apply Lemma~\ref{lem2.3} with $K=M$ which gives
\[
\#\Big\{S\le s<S+M:\|(s-q)\theta\|\le 2^{2+\tau_{\psi}-\varepsilon} SM^{-\tau_{\psi}+\varepsilon}\Big\}\ll SM^{1-\tau_{\psi}+\varepsilon}+\frac{M}{Q_\theta(M)}.
\]
Each corresponding intersection $A_\eta(p,q)\cap D_\eta(r,s)$ has length at most
\[
\ll \frac{\psi(s)}s\le\frac{\psi(S)}S.
\]
Thus, for any fixed $q\in \mathcal{Q}$,
\begin{align*}
&\sum_{p\bmod q}\sum_{\substack{M<s\le M^{\tau_{\psi}-2\varepsilon}\\ q\nmid s}}\sum_{\mathbf r\bmod s}\mathcal L_m\left(A_\eta( p,q)\cap D_\eta(r,s)\right)
\\ \ll &\sum_{\mathrm{blocks}} \left( SM^{1-\tau_{\psi}+\varepsilon}+\frac{M}{Q_\theta(M)}\right)\cdot \left(\frac{\psi(S)}S\right)\\
\le &\sum_{1\le i\le M^{\tau_{\psi}-2\varepsilon-1}}\left(\psi(iM)M^{1-\tau_{\psi}+\varepsilon}+\frac{M}{Q_{\theta}(M)}\cdot \frac{\psi(iM)}{iM}\right)
\end{align*}
By \eqref{eq:1-regular}, one has
\[
\psi(iM)\le\psi(M)\left(\frac 1i\right)^{\tau_{\psi}-\varepsilon}.
\]
The sum over $i$ is convergent, so
\begin{align*}
&\sum_{p\bmod q}\sum_{\substack{M<s\le M^{\tau_{\psi}-2\varepsilon}\\ q\nmid s}}\sum_{\mathbf r\bmod s}\mathcal L_m\left(A_\eta( p,q)\cap D_\eta(r,s)\right)\\
\ll &\psi(M)\left(M^{1-\tau_{\psi}+\varepsilon}+\frac{1}{Q_{\theta}(M)}\right)=o(\psi(M)).
\end{align*}
Hence, summing over $q\in\mathcal Q_M$, the contribution to $\mathcal I_M^{(2)}(\eta)$ of this range is also $o(S_M(\eta)).$
    \end{itemize}
\vskip -17pt\end{proof}

We now convert the measure estimate on $\mathcal I_M(\eta)$ into a counting estimate for exact cubes which fail to be $(M,\eta,\delta)$-good.
\begin{proposition} As $M\to\infty$, the number of exact cubes in $\mathscr C_M^\eta$ which are not $(M,\eta,\delta)$-good is
\[
o\left(\frac{M^{m+1}}{\log M}\right).
\]
In particular, for any fixed $L>0$, this number is
\[
o\left(L^m\frac{M^{m+1}}{\log M}\right).
\]
\end{proposition}

\begin{proof}
Combining the last two propositions together, we obtain
\begin{equation}\label{eq:14-overlap}
\mathcal I_M(\eta)=o(S_M(\eta)).
\end{equation}

If an exact cube $\mathcal C\in\mathscr C_M^\eta$ is not $(M,\eta,\delta)$-good, then
\[
\mathfrak B_M^\eta(\mathcal C)>\delta\mathcal L_m(\mathcal C)=\delta\ell_M(\eta)^m.
\]
Therefore the number of exact cubes in $\mathscr C_M^\eta$ which are not $(M,\eta,\delta)$-good is at most
\[
\frac{\mathcal I_M(\eta)}{\delta\ell_M(\eta)^m}.
\]
By \eqref{eq:14-overlap} and \eqref{eq:14-mass}, it follows
\[
\frac{\mathcal I_M(\eta)}{\delta\ell_M(\eta)^m}=o\left(\frac{S_M(\eta)}{\ell_M(\eta)^m}\right)=o\left(\frac{M^{m+1}}{\log M}\right).
\]
This proves the first assertion.

For any fixed $L>0$, by choosing $M$ sufficiently large, one has
\[
o\left(\frac{M^{m+1}}{\log M}\right)=o\left(L^m\frac{M^{m+1}}{\log M}\right),
\]
which proves the second assertion.
\end{proof}

\section{Completion of the proof: $n=1$}\label{s:1-proof}

We now choose the parameters recursively and construct the Cantor set. Fix $0<\delta<1,$ let $\eta_j=1-2^{-j-1}, j\ge0,$ which tends to 1 as $j\to \infty$. Choose $T_0$ sufficiently large so that $\T^m$ is $(T_0,\eta_0,\delta)$-good, and let $L_0=1,M_0=1.$

Suppose that $M_{j-1},T_{j-1},L_{j-1}$ have already been chosen, we choose $M_j$ satisfies the following,
\begin{itemize}
    \item $M_j$ is a regular scale so that Proposition~\ref{p:1-tran} applies with $$T=T_{j-1},\ \ \ \eta_0=\eta_{j-1},\ \ \  \eta=\eta_j,\ \ \ L=L_{j-1};$$
    \item $M_j$ is sufficiently large so that
    \begin{equation}\label{eq:1-scale-c1}
    M_{j-1}=o(\log M_{j}),\ \ \ 2^{j+1}=o(\log M_j),\ \ \ L_{j-1}^{-1} \le o(\log M_j).
    \end{equation}
\end{itemize}
Define $T_j=M_j, L_j=\ell_{M_j}(\eta_j).$ It is directly to see
\begin{align}\label{eq:1-scale-c2}
L_j=(1-\eta_j)\frac{\psi(M_j)}{M_j}\gg M_j^{-(1+\tau_\psi+2\varepsilon)}, \ \ {\text{and}}\ L_j
\ll M_j^{-(1+\tau_\psi-\varepsilon)}=o(M_j^{-\rho}).
\end{align}

We now construct the Cantor set. Let $\mathcal F_0=\{\T^m\}$ and suppose that $\mathcal F_{j-1}$ has already been constructed for $j\ge 1$.
Every $P\in\mathcal F_{j-1}$ is a cube of side length $L_{j-1}$ and is $(T_{j-1},\eta_{j-1},\delta)$-good.  By Proposition~\ref{p:1-tran}, inside every such parent $P$ there are
at least
\[
c_1L_{j-1}^m\frac{M_j^{m+1}}{2\log M_j}
\]
exact cubes satisfying all the conditions there. We keep exactly
\begin{equation}\label{eq:1-child-num}
b_j:=\left\lfloor c_1L_{j-1}^m\frac{M_j^{m+1}}{2\log M_j}\right\rfloor	
\end{equation}
of them and denote by $\mathcal F_j(P)$ those children of $P$. Let
\[
\mathcal F_j=\bigcup_{P\in\mathcal F_{j-1}}\mathcal F_j(P).
\]
All cubes in $\mathcal{F}_j$ have the same side length $L_j$ and are $(T_j,\eta_j,\delta)$-good, and those of the same parent
are pairwise $\gg M_j^{-\rho}$ separated.
Define
\begin{equation*}
	K=\bigcap_{j=0}^\infty\bigcup_{ Q\in\mathcal F_j}Q.
\end{equation*}

\begin{proposition}\label{p:1-subset} Let $K$ be defined as above, we have
\[
K\subset\Exact_{m,1}(\psi,\boldsymbol\theta).
\]
\end{proposition}

\begin{proof}
Let $\mathbf x\in K$.  For each $j\ge 1$, the point $\mathbf x$ belongs to some exact cube $A_{\eta_j}(\mathbf p_j,q_j)\text{ with } q_j\in\mathcal Q_{M_j}.$ Therefore, one has $\bx\in W_{m,1}(\psi,\bta)$.

For any $0<c<1$, choose $j_0$ large that $\eta_{j-1}>$ for all $j\ge j_0.$
For every integer $s>T_{j_0}$, there exists $j\ge j_0+1$ such that $T_{j-1}<s\le T_j.$
By the construction of $\mathcal{F}_j$, every cube in $\mathcal{F}_j$ is disjoint from
$D_{\eta_{j-1}}(\mathbf r,s){\text{ for all }} T_{j-1}<s\le T_j,\mathbf r\in\Z^m.$
Hence, since $\mathbf x$ belongs to some cube in $\mathcal{F}_j$, we have
\[
|s\mathbf x-\mathbf r-\boldsymbol\theta|\ge\eta_{j-1}\psi(s)>c\psi(s), \text{ for every } s\ge T_{j_0} \text{ and } \mathbf r\in\Z^m.\qedhere
\]
\end{proof}

It remains to estimate the Hausdorff dimension of $K$.

\begin{proof}[Proof of Theorem \ref{t:main}: the case $n=1$] We verify that the Cantor set $K$ constructed above satisfies the hypotheses of Lemma~\ref{l:cantor}.

At level $j$, the relevant parameters are side length $L_j$, child numbder $b_j$ and separation $M_j^{-\rho}$. The separation condition in Lemma~\ref{l:cantor}\textup{(1)} is automatically satisfied by the construction of $K$. We therefore verify conditions \textup{(2)} and \textup{(3)}.

\begin{itemize}
\item Condition \textup{(2)} of Lemma~\ref{l:cantor}.
For each $P\in\mathcal F_{j-1},$ and let $I\subset\T^m$ be a cube of side length $r<L_{j-1}.$ For any fixed $q\in\mathcal Q_{M_j}$, the resonant points $\{R_{\mathbf p,q}: \mathbf p\bmod q\}$ are $q^{-1}$-separated, thus a volume estimate shows that
 \[
\#\Big\{A_{\eta_j}(\mathbf p,q): A_{\eta_j}(\mathbf p,q)\cap I\ne \emptyset,\  \mathbf p\bmod q\Big\}\ll(qr+1)^m\ll(M_jr+1)^m.
\]
Summing over $q\in\mathcal Q_{M_j}$, one has
$$\#\Big\{Q\in\mathcal F_j(P): Q\cap I\ne\emptyset\Big\}\ll \frac{M_j}{\log M_j}(M_jr+1)^m\asymp b_j\left(\frac{r+M_j^{-1}}{L_{j-1}}\right)^m.$$

	 \item Condition \textup{(3)} of Lemma~\ref{l:cantor}.
By the definitions of $r_j$ and $b_j$ (see \eqref{eq:1-child-num}), we have
\begin{align}\label{eq:1-rj}
r_j=\frac{b_j^{1/m}M_j^{-(1+\rho)}}{L_{j-1}}\asymp M_j^{1/m-\rho}(\log M_j)^{-1/m}.
\end{align}
Since $\rho>1+1/m$ and $m>0$, we have $r_j=o(M_j^{-1})$, $M_j^{-\rho}=o(r_j).$
Moreover, together with \eqref{eq:1-scale-c2}, we have
\[
M_j^{-1-\tau_{\psi}-2\varepsilon}\ll L_j \ll M_j^{-1-\tau_{\psi}+\varepsilon},
\] it follows trivially that, for all sufficiently large $j$,
\begin{equation*}
L_j\le M_j^{-\rho}\le r_j\le M_j^{-1}<L_{j-1}.
\end{equation*}
\end{itemize}

Now we estimate the two liminfs appearing in the conclusion of Lemma~\ref{l:cantor}.
\begin{itemize}
\item
By \eqref{eq:1-child-num} and \eqref{eq:1-scale-c1},
\[
\log b_j =m\log L_{j-1}+ (m+1)\log M_j-\log\log M_j+O(1)=(m+1-o(1))\log M_j.
\]
By \eqref{eq:1-scale-c2},
\[
-\log L_j\le (1+\tau_\psi+2\varepsilon)\log M_j+O(1),
\]
it then follows  that
\begin{equation}\label{eq:1-first-liminf}
\liminf_{j\to\infty}\frac{\log B_j}{-\log L_j}\ge\frac{m+1}{1+\tau_\psi+2\varepsilon}.
\end{equation}

	 \item
By \eqref{eq:1-scale-c1},
\begin{align*}
	 	\log B_{j-1}+m\log(L_{j-1}M_j)&\ge m\log M_j+m\log L_{j-1}=(m+o(1))\log M_j.
	 \end{align*}
On the other hand, \eqref{eq:1-rj} gives
\begin{equation*}
-\log r_j=\left(\rho-\frac1m+o(1)\right)\log M_j.
\end{equation*}
Hence
\begin{equation}\label{eq:1-second-liminf}
\liminf_{j\to\infty}\frac{\log B_{j-1}+m\log(L_{j-1}M_j)}{-\log r_j}\ge\frac{m}{\rho-1/m},
\end{equation}
which can be very close to $m$ if we let $\rho\to 1+1/m$.
\end{itemize}

Thus Lemma~\ref{l:cantor}, together with \eqref{eq:1-first-liminf} and \eqref{eq:1-second-liminf}, yields
\begin{equation*}
\hdim K\ge\frac{m+1}{1+\tau_\psi+2\varepsilon}.
\end{equation*}

Since $K$ is a subset $\Exact_{m,1}(\psi,\boldsymbol\theta)$, we obtain
\[
\hdim\Exact_{m,1}(\psi,\boldsymbol\theta)\ge\frac{m+1}{1+\tau_\psi+2\varepsilon}.
\]
Letting $\varepsilon\to0$, the desired result follows.
\end{proof}

\section{Preliminaries for Case II: $n\ge2$}\label{s:2-case}

Throughout this section, $n\ge 2$ and $m\ge 1$. Recall the condition that $\tau_{\psi}>n/m$ which ensures
\begin{equation}\label{eq:2-sum-ve}
\sum_{q=1}^{\infty}{q^{n-1}}\psi(q)^m<\infty\qaq m(\tau_\psi-\varepsilon)>n,
\end{equation} for small $\varepsilon>0$.

Our strategy is to construct suitable subsets of $\Exact_{m,n}(\psi,\boldsymbol\theta)$ inside a family of vertical fibers: for $U\in\Mat_{m,n-1}(\T)$,
\begin{align*}\Big\{(U,\mathbf z)\in \Mat_{m,n}(\T): &\ \|U\bu+q\mathbf z-\boldsymbol\theta\|<\psi(|(\bu, q)|), {\text{ i.m.}}\ (\bu, q)\in \mathbb{Z}^n,\\
\ {\text{for any}} \ 0<c&<1, \ \ \|U\bu+q\mathbf z-\boldsymbol\theta\|>c\psi(|(\bu, q)|), \ {\text{for all}}\ (\bu, q), \ {\text{with}}\ |(\bu, q)|\gg 1\Big\}.\end{align*}

Once a fiber $U$ is fixed, the resulting approximation problem is reduced to inhomogeneous $n=1$ case, so many similar arguments. However, the reduction to a fiber does not allow us to apply the $n=1$ argument verbatim: the number of the resonant points
(\ref{eq:2-res-def}) will be increased because of the freedom of $\mathbf u$ and the counting estimates now depend on the chosen
fiber.

After obtaining a uniform lower bound for the Hausdorff dimension on a positive-measure family of fibers, a lower bound for $\hdim \Exact_{m,n}(\psi,\boldsymbol\theta)$ follows by the slicing lemma.

\begin{lemma}[Slicing lemma {\cite[Lemma 4]{BerV}}]\label{l:2-slicing} Let $k,\ell\in\N$ with $\ell<k$, and let $A\subset\R^k$ be a Borel set. For $x\in\R^\ell$, write
\[
A_x:=\left\{y\in\R^{k-\ell}:(x,y)\in A\right\}.
\]
If
$\hdim A_x\ge s$
for a set of $x\in\R^\ell$ of positive $\mathcal H^\ell$-measure, then
\[
\hdim A\ge \ell+s.
\]
\end{lemma}

Recall Lemma~\ref{lem2.1}: there are infinitely many regular scales $M$
in the sense that
\begin{equation*}
	\psi(M)\ge M^{-(\tau_\psi+\varepsilon)}\ \
{\text{and}}\ \ \psi(t)\le \psi(M)\left(\frac{t}{M}\right)^{-(\tau_\psi-\varepsilon)}, \text{ for all } t\ge M;
\end{equation*}
\begin{align}\label{eq:2-regular}
\psi(kq)\le \psi(M)\left(\frac{kq}{M}\right)^{-(\tau_\psi-\varepsilon)}\le 2^{\tau_\psi-\varepsilon}\psi(M)k^{-(\tau_\psi-\varepsilon)}, \ \ \text{ for all } \frac M2\le q\le M, \ k\ge 2;
\end{align}
\begin{equation*}
\psi(t)\le t^{-(\tau_\psi-\varepsilon)}\text{ for all }t\ge M.
\end{equation*}
As before, all the integers $M$, $M_j$ below are chosen from the regular scales.

\subsection{Fiber resonant points and exact cubes}

For any $A\in\Mat_{m,n}(\T)$, write
\[A=(U,\mathbf z)\qquad\text{with $U\in\Mat_{m,n-1}(\T)$ and $\mathbf z\in\T^m$}.\]
For an integer vector $\ba\in\Z^n$, we view it as a column vector and write
\[\mathbf a=(\mathbf u,q)\qquad\text{with $\mathbf u\in\Z^{n-1}$ and $q\in\Z.$}\]
For convenience, we call $\mathbf a\in\mathbb Z^n$ a \emph{denominator vector} and refer to its norm $|\mathbf a|$ as its \emph{height}.
 With these conventions,
\[
A\mathbf a=A(\bu, q)=U\mathbf u+q\mathbf z.
\]
For fixed $U$ and $q\ne0$, the equation
$U\mathbf u+q\mathbf z-\mathbf p-\boldsymbol\theta=0$
gives \emph{fiber resonant point}
\begin{equation}\label{eq:2-res-def}
R_{\mathbf u,\mathbf p,q}(U)=\frac{\mathbf p+\boldsymbol\theta-U\mathbf u}{q}\pmod{\Z^m}.
\end{equation}

We fix some notations which are similar to $n=1$ case.
\begin{itemize}\item For a regular scale $M\ge M_\ve$, define
\begin{equation*}
\mathcal Q_M:=\left\{(\bu,q)\in\Z^n: 0<|\mathbf u|\le M/4,\ M/2\le q\le M \text{ and }q\text{ is prime}\right\}.
\end{equation*}
For each {\em denominator vector} $\ba=(\bu,q)\in\mathcal Q_M$, we have $|\mathbf a|=q$ and there is no nontrivial common divisor of $\bu$ and $q$, since $q$ is prime and $|\bu|<q$. It is direct to see that
\begin{equation*}
\#\mathcal Q_M \asymp \frac{M^n}{\log M}.
\end{equation*}

\item For any $0<\eta<1$, as before let
\begin{equation*}
\ell_M(\eta)=(1-\eta)\frac{\psi(M)}M.
\end{equation*}

\item For $\mathbf a=(\mathbf u,q)\in\mathcal Q_M$ and $\mathbf p\in\Z^m$, define the \emph{fiber exact cube}
\begin{equation*}
A_\eta(\mathbf u,\mathbf p,q;U)=R_{\mathbf u,\mathbf p,q}(U)+\left[\eta\frac{\psi(q)}q,\,\eta\frac{\psi(q)}q+\ell_M(\eta)\right]^m.
\end{equation*}

Then, for every $\mathbf z\in A_\eta(\mathbf u,\mathbf p,q;U)$,
\begin{equation*}
	\eta\psi(q)\le|U\mathbf u+q\mathbf z-\mathbf p-\boldsymbol\theta|\le \eta\psi(q)+(1-\eta)\frac{\psi(M)}M<\psi(q).
\end{equation*}
\item
For $U\in\Mat_{m,n-1}(\T)$, define
\[
\mathscr R_M(U) :=\left\{R_{\mathbf u,\mathbf p,q}(U): (\mathbf u,q)\in\mathcal Q_M,\ \mathbf p\bmod q\right\},
\]
a collection of fiber resonant points and let
\[
\mathscr C_M^\eta(U):=\left\{A_\eta(\mathbf u,\mathbf p,q;U): (\mathbf u,q)\in\mathcal Q_M,\ \mathbf p\bmod q\right\}
\]
be the corresponding collection of fiber exact cubes.

\item For $\mathbf a,\mathbf b\in\Z^n\setminus\{\mathbf0\}$, define
$$\mathcal E_\eta(\mathbf a)=\left\{A\in\Mat_{m,n}(\T): A\mathbf a-\boldsymbol\theta\in\left[\eta\psi(q),\,\eta\psi(q)+q\ell_M(\eta)\right]^m \pmod{\Z^m}\right\},$$
$$
\mathcal D_\eta(\mathbf b)=\left\{A\in\Mat_{m,n}(\T):\|A\mathbf b-\boldsymbol\theta\|<\eta\psi(|\mathbf b|)\right\}.
$$
\end{itemize}

\subsection{Fiber dangerous cubes}

For any $\mathbf b=(\mathbf v,s)\in \Z^{n-1}\times\Z\setminus\{\mathbf0\}$, $\mathbf r\in\Z^m$ and $U\in\Mat_{m,n-1}(\T)$, define
\begin{equation}\label{eq:2-dan-def}
D_\eta(\mathbf b,\mathbf r;U)=B\left(\frac{\mathbf r+\boldsymbol\theta-U\mathbf v}{s},\eta\frac{\psi(|\bb|)}{|s|}\right).
\end{equation}
For obvious reasons, we call $D_\eta(\mathbf b,\mathbf r;U)$ a \emph{fiber dangerous cube}.

These are not all the dangerous parts that we should avoid, since we need also to avoid the hyperplane generated by $\mathbf b=(\mathbf v,s)$ with $s=0$. In this case, the part that should be avoided is
\begin{equation*}
\|U\mathbf v-\boldsymbol\theta\|<\eta\psi(|\mathbf v|).
\end{equation*}
The next lemma says that for almost all $U$, this can be avoided trivially.
\begin{lemma}\label{l:2-horizontal} For $\mathcal L_{m(n-1)}$-almost all $U\in \Mat_{m,n-1}(\T)$, only for finitely many $\mathbf v\in\Z^{n-1}$
\begin{equation*}
\|U\mathbf v-\boldsymbol\theta\|<\psi(|\mathbf v|).
\end{equation*}
\end{lemma}

\begin{proof} Consider the limsup set $$
\{U\in \Mat_{m,n-1}(\T): \|U\mathbf v-\boldsymbol\theta\|<\psi(|\mathbf v|), \ {\text{i.m.}}\ \ \mathbf v\in\Z^{n-1}\}.$$
Then, \begin{align*}
&\sum_{\mathbf v\in\Z^{n-1}}\mathcal{L}_{m(n-1)}\left(\Big\{U\in\Mat_{m,n-1}(\T): \|U\mathbf v-\boldsymbol\theta\|<\psi(|\bv|)\Big\}\right)\\
\ll&\ \sum_{\mathbf v\in\Z^{n-1}}\psi(|\mathbf v|)^{m}\ll \sum_{h=1}^{\infty} h^{n-2}\psi(h)^m<\infty
\end{align*} by (\ref{eq:2-sum-ve}). Thus, the convergence part of Borel--Cantelli lemma concludes the desired result.
\end{proof}

In the following, we focus only on the full measure fibers $U$ obtained from Lemma \ref{l:2-horizontal}, so all the dangerous part which should be avoided are the cubes in (\ref{eq:2-dan-def}).

\subsection{Good fiber cubes}

 For any set $J\subset \T^m$ and a real number $T>1$, define
\begin{equation*}
\mathfrak B_T^\eta(J;U)=\sum_{\substack{\mathbf b=(\mathbf v,s)\in\Z^n\setminus\{\mathbf0\}\\ s\ne0,\ |\mathbf b|>T}}\sum_{\mathbf r\bmod s}\mathcal L_m\left(J\cap D_\eta(\mathbf b,\mathbf r;U)\right).
\end{equation*}

\begin{definition}[Good fiber cube]\label{d:2-good} Let $T>1$ and $0<\eta,\delta<1$ and $U\in\Mat_{m,n-1}(\T)$. A fiber cube $J\subset\T^m$ is called $(T,\eta,\delta)$-\emph{good with respect to $U$} if
\begin{equation*}
\mathfrak B_T^\eta(J;U)\le\delta\,\mathcal L_m(J).
\end{equation*}
\end{definition}

The existence of good fiber cubes follows from the same idea as $n=1$ case.
\begin{lemma}\label{l:2-initial} Let $U\in\Mat_{m,n-1}(\T)$ and let $J\subset\T^m$ be a fiber cube. Then, for every $\eta\in(0,1)$ and every $\delta>0$, there exists $T_0=T_0(J,\delta)>1,$ independent of $U$, such that $J$ is $(T,\eta,\delta)$-good with respect to $U$ for every $T\ge T_0$.
\end{lemma}

\begin{proof} Note that
$$\#\{\mathbf b=(\mathbf v,s)\in\Z^n\setminus\{\mathbf0\}: |\mathbf b|=h\}\asymp h^{n-1}, \
\ \mathcal{L}_m(D_\eta(\mathbf b,\mathbf r;U))\asymp \left(\frac{\psi(|\mathbf b|)}{s}\right)^m. $$
Then one has
\[
\mathfrak B_T^\eta(J;U)\le \sum_{\substack{\mathbf b=(\mathbf v,s)\in\Z^n\setminus\{\mathbf0\}\\ s\ne0,\ |\mathbf b|>T}}\sum_{\mathbf r\bmod s}
\mathcal L_m\left(D_\eta(\mathbf b,\mathbf r;U)\right)\le \sum_{h> T}h^{n-1}\psi(h)^m=o(1),
	\]
as $T\to \infty$, by the convergence of the last series, see (\ref{eq:2-sum-ve}).

The estimation is independent on $U$. So once $T$ is sufficiently large independent on $U$, one has
\[\mathfrak B_T^\eta(J;U)\le\delta\,\mathcal L_m(J).\qedhere\]
\end{proof}

\subsection{Auxiliary lemmas} The next several lemmas characterize the intersection properties of fiber exact cubes and fiber dangerous cubes
which are similar to $n=1$ case.

\begin{lemma}\label{l:2-resonant-1}
Suppose that
\[
\|R_{\bu,\bp,q}(U)-R_{\bv,\br,s}(U)\|<\delta.
\]
Then there exists $\bk\in\Z^m$ such that
\begin{equation}\label{eq:2-k-exists}
\left|\bk+(s-q)\boldsymbol\theta-U(s\bu-q\bv)\right|<q|s|\delta
\end{equation}
and
\begin{equation}\label{eq:2-k-mod}
\bk\equiv s\bp\pmod q,\ \ \ \ \bk\equiv-q\br\pmod{|s|},\ \ \ \ \bk\equiv s\bp-q\br\pmod{q|s|}.
\end{equation}
In particular, writing $d=(q,s)$, we have $\bk\in d\,\Z^m,$
and each $\bk$ determines exactly $d^m$ distinct pairs
$(\bp\bmod q,\br\bmod s)$.
\end{lemma}

\begin{proof}
Applying Lemma~\ref{l:1-resonant-1} with
$\bta=U(s\bu-q\bv)$, we obtain an integer vector $\bk$
satisfying \eqref{eq:2-k-exists} and \eqref{eq:2-k-mod}, given by
\[
\bk=s\bp-q\br+qs\boldsymbol\ell
\]
for some $\boldsymbol\ell\in\Z^m$.

Let $d=(q,s)$. Since $d$ divides both $q$ and $s$, we have $\bk\in d\Z^m.$
For each coordinate, the congruence
\[
k_i\equiv sp_i\pmod q
\]
has exactly $d$ solutions $p_i\bmod q$. Once $p_i$ is fixed, the last
congruence in \eqref{eq:2-k-mod} gives
\[
r_i\equiv \frac{sp_i-k}{q}\pmod{|s|},
\]
so $r_i\bmod s$ is uniquely determined. Therefore,
each $\bk$ determines exactly $d^m$ distinct pairs
$(\bp\bmod q,\br\bmod s)$.
\end{proof}

\begin{lemma}\label{l:2-resonant-2}
Fix $0<\eta_0<\eta<1$ and an integer $K\ge1$. Then, for all sufficiently large $M$, for every $\ba=(\bu,q)\in\mathcal Q_M$,
$$
\mathcal E_{\eta}(\ba)\cap\mathcal D_{\eta}(k\ba)=\emptyset
$$
for every integer $k$ with $1\le |k|\le K$. In particular, for any $1\le |k|\le K$ and any such $M$, we have
$$
A_{\eta}(\bu,\bp,q;U)\cap D_{\eta_0}(k\ba,\br;U)=\emptyset
$$
for every $U\in\Mat_{m,n-1}(\T)$, $\bp\bmod q$ and $\br\mod s$.
\end{lemma}
\begin{proof}
Assume by contradiction that
$$
A\in\mathcal E_{\eta}(\ba)\cap\mathcal D_{\eta}(k\ba)
$$
for some integer $k$ with $1\le |k|\le K$. Then there exists $\mathbf p\in\Z^m$ such that
$$
\mathbf t:=A\mathbf a-\mathbf p-\boldsymbol\theta
\in\left[\eta\psi(q),\,\eta\psi(q)+q\ell_M(\eta)\right]^m
\qaq
\|A(k\mathbf a)-\boldsymbol\theta\|<\eta\psi(|k|q).
$$

Using
$
A\mathbf a=\mathbf p+\boldsymbol\theta+\mathbf t
$
in the second inequality, we obtain
\begin{equation}\label{eq:2-mul-c}
|k\mathbf t+(k-1)\boldsymbol\theta|<\eta\psi(|k|q).
\end{equation}
This is exactly the same situation as in \eqref{eq:1-mul-c}, and hence the same argument shows that \eqref{eq:2-mul-c} cannot hold once $M$ is sufficiently large. Therefore, the intersection $
\mathcal E_{\eta}(\ba)\cap\mathcal D_{\eta}(k\ba)
$ is empty.

The ``In particular" part follows immediately by noting that
\[A_{\eta}(\bu,\bp,q;U)\cap D_{\eta_0}(k\ba,\br;U)
\subset
\mathcal E_{\eta}(\ba)\cap\mathcal D_{\eta}(k\ba)
=\emptyset,\]
since $\eta_0<\eta$.
\end{proof}

We cite a standard result about the invariance of Haar measure under continuous surjective homomorphisms.
\begin{lemma}\label{l:2-haar-push} Let $G_1$ and $G_2$ be compact abelian groups equipped with normalized Haar measures $\mu_1$ and $\mu_2$, respectively. Let
$\pi:G_1\to G_2$
be a continuous surjective group homomorphism. Then for every $\boldsymbol\xi\in G_2$
and every Borel set $B\subset G_2$,
\[
\mu_1\left(\left\{x\in G_1: \pi(x)-\boldsymbol\xi\in B\right\}\right)=\mu_2(B).
\]
\end{lemma}
\begin{proof}
By \cite[Lemma~2.7]{LeLe2025}, we have
$\pi_*\mu_1=\mu_2.$
The conclusion follows from the translation invariance of $\mu_2$.
\end{proof}

This leads to the following lemma, simple but vital for counting the integer vectors $\mathbf k$.
\begin{lemma}\label{l:2-haar-ave}
Let $d\in\Z\setminus\{0\}$. Let
$\mathbf w\in d\,\Z^{n-1}\setminus\{\mathbf 0\}$,
$\boldsymbol\gamma\in\R^m$ and $\ve>0$. Then,
\begin{equation}\label{eq:2-haar-ave}
\int_{\T^{m(n-1)}}
\#\left\{
\mathbf k\in d\,\Z^m:
|\mathbf k+\boldsymbol\gamma-U\mathbf w|<\ve
\right\}
\dif\mathcal L_{m(n-1)}(U)
=
\left(\frac{2\ve}{|d|}\right)^m.
\end{equation}
\end{lemma}

\begin{proof}
Since $\mathbf w\in d\,\Z^{n-1}\setminus\{\mathbf 0\}$, we have
\[
\frac{\mathbf w}{d}\in\Z^{n-1}\setminus\{\mathbf 0\}.
\]
Hence, the homomorphism
\[
U\longmapsto U\frac{\mathbf w}{d}\pmod{\Z^m}
\]
is onto $\T^m$. Writing $\mathbf k=d\mathbf j$, the left-hand side of
\eqref{eq:2-haar-ave} is equal to
\begin{align*}
&\int_{\T^{m(n-1)}}
\#\left\{
\mathbf j\in\Z^m:
\left|
\mathbf j+\frac{\boldsymbol\gamma}{d}
-U\frac{\mathbf w}{d}
\right|
<
\frac{\varepsilon}{|d|}
\right\}
\dif\mathcal L_{m(n-1)}(U)\\
=&\int_{\T^m}
\#\left\{
\mathbf j\in\Z^m:
\left|
\mathbf j+\frac{\boldsymbol\gamma}{d}-\mathbf x
\right|
<
\frac{\varepsilon}{|d|}
\right\}
\,\dif\mathcal L_m(\mathbf x)\\
=&\int_{[0,1)^m}
\sum_{\mathbf j\in\Z^m}
\mathbf 1_{\left(-\frac{\varepsilon}{|d|},
\frac{\varepsilon}{|d|}\right)^m}
\left(
\mathbf x-\mathbf j-\frac{\boldsymbol\gamma}{d}
\right)
\dif\mathbf x\\
=&\int_{\R^m}
\mathbf 1_{\left(-\frac{\varepsilon}{|d|},
\frac{\varepsilon}{|d|}\right)^m}(\mathbf y)
\dif\mathbf y
=
\left(\frac{2\varepsilon}{|d|}\right)^m.
\end{align*}
This completes the proof of \eqref{eq:2-haar-ave}.
\end{proof}

\begin{lemma}\label{l:2-ave-count}
Let $\ba=(\bu,q),\bb=(\bv,s)\in\Z^n$
with $q,s\ne0$, and put
\[
d=(q,s),\qquad \bw=s\bu-q\bv.
\]
Assume that $\bw\ne0$. For  $\delta>0$, let
\[
N_{\ba,\bb}(U;\delta)
:=
\#\left\{
\bp\bmod q:
\|R_{\bu,\bp,q}(U)-R_{\bv,\br,s}(U)\|<\delta\text{ for some }\mathbf r \bmod s
\right\}.
\]
Then,
\[
\int N_{\ba,\bb}(U;\delta)
\dif\mathcal L_{m(n-1)}(U)
\le
(2q|s|\delta)^m.
\]
\end{lemma}

\begin{proof}
By Lemma~\ref{l:2-resonant-1}, every pair counted by
$N_{\ba,\bb}(U;\delta)$ gives rise to some $\bk\in d\Z^m$ satisfying
\[
|\bk+(s-q)\boldsymbol\theta-U\bw|
<q|s|\delta.
\]
Moreover, each such $\bk$ determines exactly $d^m$ pairs
$(\bp\bmod q,\br\bmod s)$. Hence,
\[
N_{\ba,\bb}(U;\delta)
\le
d^m
\#\left\{
\bk\in d\Z^m:
|\bk+(s-q)\boldsymbol\theta-U\bw|
<q|s|\delta
\right\}.
\]
Since $d\mid\bw$ and $\bw\ne0$, Lemma~\ref{l:2-haar-ave} gives
\[
\begin{aligned}
\int N_{\ba,\bb}(U;\delta)
\dif\mathcal L_{m(n-1)}(U)
\le
d^m
\left(\frac{2q|s|\delta}{d}\right)^m=
(2q|s|\delta)^m.\qedhere
\end{aligned}
\]
\end{proof}

\subsection{Main proposition}

Let $\cu$ be the full-measure set given by Lemma~\ref{l:2-horizontal}. By Egorov's theorem, by discarding a set of arbitrarily small measure, we can find a set
$\Omega \subset \Mat_{m,n-1}(\T)\cap \cu$ such that for every $U\in\Omega$, the inequality
\begin{equation}\label{eq:2-no-solu}
|U\mathbf v-\mathbf r-\boldsymbol\theta|<\psi(|\mathbf v|)
\end{equation}
has no solution $(\mathbf v,\mathbf r)\in \Z^{n-1}\times\Z^m$ with $|\mathbf v|>T$.

 Fix parameters $T>1$, $0<\eta_0<1$, $L>0$ and $0<\delta<1$. Fix $U\in\Mat_{m,n-1}(\T)$ and let $J$ be a $(T,\eta_0,\delta)$-good cube with respect to $U$ whose existence is given by Lemma \ref{l:2-initial}. It is difficult to determine directly how many exact cubes inside $J$ have the desired properties required.
 The difficulty comes from the fact that we have no information on $U$.

So our approach is to estimate the number of exact cubes which fail the desired conditions on average over $U$. Once this averaged number of not good cubes is small, we can conclude for a large family of fibers $U$, most of the cubes satisfies the desired conditions. The main reason that this can be achieved is because of the convergence of the series $$
\sum_{q\ge 1}q^{n-1}\psi(q)^m.
$$

Let $\rho$ be a positive number very close to $1+n/m$ so it can satisfy
\begin{equation}\label{eq:2-rho}
1+\frac nm<\rho<\min\left\{1+\tau_\psi-\varepsilon,\, \frac nm+\frac{m(1+\tau_\psi+2\varepsilon)}{m+n}\right\}.
\end{equation}

The estimates developed below show that this strategy is effective: the set of discarded parameters has small measure.

\begin{proposition}\label{p:2-tran} Let $T,\eta_0,L$ and $\delta$ be chosen as above. Let $\eta_0<\eta<1$ and $\xi>0$. Choose $a>0$ so that $a(m+n)<1$.
Then, for every sufficiently large regular scale $M$ with $M^a>T$, there exists a measurable exceptional set $\mathcal X\subset\Omega,$ depending on $L$, such that
\[
\mathcal L_{m(n-1)}(\mathcal X)\le 3\xi.
\]
Moreover, there is a constant $c_2=c_2(m,n,\delta)>0$ such that for every $U\in\Omega\setminus\mathcal X$ and every cube $J\subset\T^m$ of side length $L$ which is $(T,\eta_0,\delta)$-good with respect to $U$, one can retain at least
\begin{equation*}
\frac{c_2L^mM^{m+n}}{2\log M}
\end{equation*}
fiber exact cubes in $\mathscr C_M^\eta(U)$ with the following properties:
\begin{enumerate}
\item[(a)] they are disjoint from every dangerous cube $D_{\eta_0}(\mathbf b,\mathbf r;U),\bb=(\bv,s)$ with $s\neq 0, T<|\mathbf b|\le M, \mathbf r\in\Z^m;$
\item[(b)] they are pairwise $\gg M^{-\rho}$-separated;
\item[(c)] all of them are $(M,\eta,\delta)$-good with respect to $U$.
\end{enumerate}
\end{proposition}

\begin{remark}
We emphasize that, for each $U\in\Omega\setminus\mathcal X$, the conclusion of Proposition~\ref{p:2-tran} holds for all $(T,\eta_0,\delta)$-good cubes of side length $L$. This uniformity is essential for the subsequent Cantor construction. Indeed, at each stage of the construction there are many parent cubes, and the same $U$ must allow the construction to be carried out simultaneously inside all of them.  The independence of the exceptional set from the parent cube is precisely what makes this recursive construction possible.
\end{remark}

We prove Proposition~\ref{p:2-tran} in the next several sections. The argument proceeds by successively removing those exact cubes which fail the conditions required in the Cantor construction, according to the height $|\bb|$ of the corresponding dangerous resonant points.

\section{Proof of Proposition \ref{p:2-tran}($a$)}

\subsection{Small heights}We first eliminate the fiber exact cubes that can intersect fiber dangerous cubes of relatively small heights.
For any $U\in\Mat_{m,n-1}(\T)$, let
\[
\mathcal V(U)=\bigcup_{\substack{T<|\mathbf b|<M^a\\ s\ne0}}\ \bigcup_{\mathbf r\in\Z^m}D_{\eta_0}(\mathbf b,\mathbf r;U),
\]
where $T$ and $a$ are given as in Proposition~\ref{p:2-tran}.

The role of the goodness assumption on $J$ ensures that the dangerous cubes occupy only a controlled proportion of the parent cube, leaving sufficiently much room for the exact cubes.

\begin{proposition}\label{p:2-low} Let $U\in\Mat_{m,n-1}(\T)$, let $J\subset\T^m$ be a cube of side length $L$ which is $(T,\eta_0,\delta)$-\emph{good with respect to $U$}, and let $a>0$ be as chosen above with $a(m+n)<1$. Then, there exists a constant $c_2=c_2(m,n,\delta)>0$ such that, for every sufficiently large $M$ satisfying $M^a>T$ and $M^{a(m+n)-1}=o(L^m)$, at least
\begin{equation}\label{eq:21-surv}
\frac{c_2L^mM^{m+n}}{\log M}	
\end{equation}
exact cubes $A_\eta(\mathbf u,\mathbf p,q;U)\in\mathscr C_M^\eta(U)$ satisfying
\[
A_{\eta}(\mathbf u,\mathbf p,q;U)\subset J\setminus\mathcal V(U).
\]
\end{proposition}
\begin{proof}
The proof follows the same argument as Proposition~\ref{p:1-low}, so we only indicate the differences in the counting. In the present setting, the dangerous cubes are indexed by $\bb=(\bv,s)$, with the additional parameter $\bv\in\mathbb Z^{n-1}$. For $|\bb|<M^a$, there are $O(M^{a(n-1)})$ possible choices for $\bv$, and hence the number of dangerous cubes under consideration is
$$
O\left(M^{a(n-1)}\sum_{s<M^a}s^m\right)
=O\left(M^{a(m+n)}\right).
$$
The total measure of these dangerous cubes inside $J$ is controlled by the $(T,\eta_0,\delta)$-goodness of $J$, exactly as in Proposition~\ref{p:1-low}.

Now fix $\bu$ with $0<|\bu|\le M/4$. For each prime $M/2\le q\le M$, the fiber resonant points
$$
R_{\bu,\bp,q}(U),\qquad \bp\bmod q,
$$
form a translate of the $q^{-1}$-grid in $\T^m$. Applying the same counting argument as in Proposition~\ref{p:1-low}, the number of surviving fiber exact cubes corresponding to this fixed $\bu$ is
$$
\gg_\delta \frac{L^mM^{m+1}}{\log M},
$$
up to an error term of order
$$
O\left(\frac{M^{a(m+n)+m}}{\log M}\right).
$$
Here the exponent $a(m+n)$, instead of $a(m+1)$ in Proposition~\ref{p:1-low}, comes from the additional $O(M^{a(n-1)})$ choices for $\bv$. By the choice $a(m+n)<1$ and the assumption
$$
M^{a(m+n)-1}=o(L^m),
$$
this error term is negligible compared with the main term.

Finally, there are $\asymp M^{n-1}$ choices for $\bu$. Summing over $\bu$, we obtain
$$
\gg_{m,n,\delta}M^{n-1}\frac{L^mM^{m+1}}{\log M}=\frac{L^mM^{m+n}}{\log M}
$$
surviving exact cubes. Choosing $c_2=c_2(m,n,\delta)>0$ sufficiently small proves \eqref{eq:21-surv}.
\end{proof}

\subsection{Intermediate heights}

In this subsection, we count the exact cubes in $\mathscr{C}_M^\eta(U)$ that have nonempty intersection with dangerous cubes corresponding to denominator vectors $\bb\in\Z^n$ satisfying
\begin{equation*}
M^a\le |\bb|\le M.
\end{equation*}

For any $\mathbf a=(\mathbf u,q)\in\mathcal Q_M$ and $\mathbf b=(\mathbf v,s)\in\Z^n$ with $s\ne 0$, let
\begin{equation*}
N_{\mathbf a,\mathbf b}(U)=\#\left\{\mathbf p\bmod q: A_{\eta }(\mathbf u,\mathbf p,q;U)\cap D_{\eta_0}(\mathbf b,\mathbf r;U)\ne\emptyset, \ \ {\text{for some}}\ \mathbf r \bmod s\right\}.
\end{equation*}
Note that the count is taken over the whole fiber $\T^m$ and is not restricted to a parent cube $J$.  In particular, $N_{\mathbf a,\mathbf b}(U)$ is independent of $J$.

Rather than estimating $N_{\mathbf a,\mathbf b}(U)$ pointwisely for each $U$, we adopt an averaged point of view by integration over $U\in \Mat_{m,n-1}(\T)$. This yields an average bound for $N_{\mathbf a,\mathbf b}(U)$, and then Markov's inequality will give us the expected bound except a set of small measure on $U$.

Define
\begin{equation*}
N_<(U):=\sum_{\substack{\mathbf b=(\mathbf v,s)\in\Z^n\\ s\ne0,\ M^a\le|\mathbf b|\le M}}\ \sum_{\mathbf a\in\mathcal Q_M}N_{\mathbf a,\mathbf b}(U).
\end{equation*}

\begin{proposition}\label{p:2-int}
Let $L$ and $\xi$ be as in Proposition~\ref{p:2-tran}, let $M$ be sufficiently large depending on the fixed parameters $L$, $\xi$ and $\delta$. Then,
\begin{equation}\label{eq:22-ave-count}
\int N_<(U)\,d\mathcal L_{m(n-1)}(U)\le C\frac{M^{m+n}}{\log M}\sum_{h\ge M^a}h^{n-1}\psi(h)^m+o\left(\frac{M^{m+n}}{\log M}\right),
\end{equation}
where $C$ depends only on $m$ and $n$. In particular, outside an exceptional set $\mathcal X_1(\xi)$ of measure at most $\xi$, we have for any $U\notin\mathcal X_1(\xi)$,
\[
N_<(U)\le\frac{c_2L^mM^{m+n}}{4\log M},
\]
 where $c_2$ is given as in \eqref{eq:21-surv}.
\end{proposition}
\begin{proof}
Fix $\mathbf a=(\mathbf u,q)\in\mathcal Q_M$, $\mathbf b=(\mathbf v,s)\in\mathbb Z^n\setminus\{0\}, $ where $M^a\le |\mathbf b|=h\le M, s\ne0.$ Let $\mathbf w=s\mathbf u-q\mathbf v.$ We estimate the integration of $N_{\mathbf a,\mathbf b}(U)$ according to
$\bw\ne0 {\text{ or }}\bw=0.$

\begin{itemize}
\item When $\mathbf w=0$. If $|s|\ne q$, then $(q,s)=1$, and the identity
$
s\mathbf u=q\mathbf v
$
would imply $q\mid\mathbf u$, a contradiction. Hence $|s|=q$, and therefore $\mathbf b=\pm\mathbf a$. By Corollary~\ref{l:2-resonant-2}, the corresponding intersection is empty and so $N_{\mathbf a,\mathbf b}(U)=0$.

\item When $\mathbf w\ne0$. Any non-empty intersection $A_{\eta}(\mathbf u,\mathbf p,q;U)\cap D_{\eta_0}(\mathbf b,\mathbf r;U)$ implies that
\[
\|R_{\bu,\bp,q}(U)-R_{\bv,\br,s}(U)\|\le\frac{\eta_0\psi(h)}{|s|}+\frac{\psi(q)}{q}.
\]
Hence, by Lemma~\ref{l:2-ave-count},
we obtain
\begin{align*}
\int N_{\mathbf a,\mathbf b}(U)\dif\mathcal L_{m(n-1)}(U)&\le\left(2q|s|\left(\frac{\eta_0\psi(h)}{|s|}+\frac{\psi(q)}{q}\right)\right)^m\ll
M^m\bigl(\psi(h)^m+\psi(q)^m\bigr).
\end{align*}
\end{itemize}

Combining the above cases and summing over all $\mathbf a$ and $\mathbf b$, we obtain
\begin{align*}
\int N_<(U)\dif\mathcal L_{m(n-1)}(U)
&\ll
\sum_{M^a\le h\le M}
h^{n-1}\frac{M^n}{\log M}
\cdot
M^m
\left(
\psi(h)^m+M^{-m(\tau_\psi-\varepsilon)}
\right)\\
&=
\frac{M^{m+n}}{\log M}
\sum_{h\ge M^a}h^{n-1}\psi(h)^m
+
o\left(\frac{M^{m+n}}{\log M}\right),
\end{align*}
where the $o(\cdot)$ term follows from $m(\tau_\psi-\varepsilon)>n$. This proves \eqref{eq:22-ave-count}.

Note that the series over $h$ converges, so by asking $M$ sufficiently large, we can ensure that
\[
\int N_<(U)\,d\mathcal L_{m(n-1)}(U)\le\frac{c_2\xi L^mM^{m+n}}{4\log M}.
\]
Then Markov's inequality is applied to conclude that the exceptional set
\[
\mathcal X_1(\xi):=\left\{U\in\Mat_{m,n-1}(\T): N_<(U)>\frac{c_2L^mM^{m+n}}{4\log M}\right\}
\]
has measure at most $\xi$.
\end{proof}

\begin{remark}
The exceptional set depends only on the side length of a cube, not on its location.
\end{remark}

\section{Proof of Proposition \ref{p:2-tran}$(b)$}

Recall the definition of $\rho$ in \eqref{eq:2-rho}. For $U\in\Mat_{m,n-1}(\T)$, denote $E_M(U)$ the collection of fiber resonant points in $\mathscr R_M(U)$ which is $M^{-\rho}$ close to some other resonant points in $\mathscr R_M(U)$, i.e.
\begin{align*}
E_M(U):=\Bigl\{R_{\bu,\bp,q}(U)&\in\mathscr R_M(U):\ \exists\,R_{\bu',\bp',q'}(U)\in\mathscr R_M(U)\text{ such that }\\ 
& R_{\bu,\bp,q}(U)\ne R_{\bu',\bp',q'}(U),\ \text{and}\ \bigl\|R_{\bu,\bp,q}(U)-R_{\bu',\bp',q'}(U)\bigr\|<M^{-\rho} \Bigr\}.
\end{align*}

\begin{proposition}\label{p:2-sep} Let $L$ and $\rho$ be as in Proposition~\ref{p:2-tran} and \eqref{eq:2-rho}, respectively. For all sufficiently large $M$ depending on $L$ and $\rho$,
\begin{equation*}
		\int \#E_M(U)\,d\mathcal L_{m(n-1)}(U)=o\left(\frac{L^mM^{m+n}}{\log M}\right).
\end{equation*}
In particular, outside an exceptional set $\mathcal X_2(\xi)$ of measure at most $\xi$, with $\xi$ as in Proposition~\ref{p:2-tran}, one may discard at most
\[
\#E_M(U)\le\frac{c_2L^mM^{m+n}}{8\log M}\qquad (U\notin\mathcal X_2(\xi))
\]
resonant points from $\mathscr R_M(U)$ so that the remaining resonant points are $M^{-\rho}$-separated, where $c_2$ is given as in \eqref{eq:21-surv}.  Consequently, the exact cubes corresponding to the remaining resonant points are pairwise $\gg M^{-\rho}$-separated.
\end{proposition}

\begin{proof} Fix $\ba=(\bu,q), \ba'=(\bu',q')\in\mathcal Q_M,$ and let $\bw=q'\bu-q\bu'$. We estimate the average number of $M^{-\rho}$-close pairs according to whether $\bw=0$ or $\bw\ne0$.

\begin{itemize}
\item When $\bw=0$. If $q\ne q'$, then $(q,q')=1$, and the identity
\[
q'\bu=q\bu'
\]
would imply $q\mid\bu$, a contradiction. Hence $q=q'$, and therefore $\bu=\bu'$. In this case, distinct resonant points are at least
$1/q$-separated, and hence cannot be $M^{-\rho}$-close for sufficiently large $M$, since $\rho>1$. Thus this case gives no contribution.

\item When $\bw\ne0$. Applying Lemma~\ref{l:2-ave-count} with $\delta=M^{-\rho},$ we obtain
\begin{align*}
&\int\#\left\{\bp\bmod q: \|R_{\bu,\bp,q}(U)-R_{\bu',\bp',q'}(U)\|<M^{-\rho}\text{ for some }\mathbf \bp' \bmod q'\right\}\dif\mathcal L_{m(n-1)}(U)\\
\le&\ (2qq'M^{-\rho})^m \ll M^{m(2-\rho)}.
\end{align*}
\end{itemize}
Summing over all $\ba,\ba'\in\mathcal Q_M,$
and using $\#\mathcal Q_M\asymp\frac{M^n}{\log M},$
we obtain
\begin{align*}
\int \#E_M(U)\dif\mathcal L_{m(n-1)}(U)&\ll \left(\frac{M^n}{\log M}\right)^2M^{m(2-\rho)}=\frac{M^{2n+2m-m\rho}}{(\log M)^2}\\
&=o\left(\frac{L^mM^{m+n}}{\log M}\right),
\end{align*}
since $\rho>1+n/m$.

Finally, Markov's inequality gives the exceptional set $\mathcal X_2(\xi)$ and the desired bound outside $\mathcal X_2(\xi)$ by the same argument as in Proposition~\ref{p:2-int}.
\end{proof}

\begin{remark}
Again, the exceptional set depends only on the side length of a cube, not on its location.
\end{remark}

\section{Proof of Proposition \ref{p:2-tran}$(c)$}\label{s:2-large}

We now turn to the goodness condition for the exact cubes. The idea for the proof is similar to $n=1$ case. The only difference is that we work on fibers and take integration over   $\Mat_{m,n-1}(\T)$ to arrive at the averaged estimation. Then Markov's inequality is used to conclude that except for a set of small measure of fibers, there are sufficiently many good cubes on fibers.

Recall the definition of $\mathcal E_\eta(\mathbf a)$ and $\mathcal D_\eta(\mathbf b)$:
\begin{align*}
\mathcal E_\eta(\mathbf a)&=\left\{A\in\Mat_{m,n}(\T): A\mathbf a-\boldsymbol\theta\in\left[\eta\psi(q),\,\eta\psi(q)+q\ell_M(\eta)\right]^m \pmod{\Z^m}\right\}\\
&=\bigcup_{\mathbf p\bmod q}\left\{(U,\mathbf z)\in\Mat_{m,n}(\T):\mathbf z\in A_\eta(\mathbf u,\mathbf p,q;U)\right\};\\
\mathcal D_\eta(\mathbf b)=\{A\in\Mat_{m,n}&(\T):\|A\mathbf b-\boldsymbol\theta\|<\eta\psi(|\mathbf b|)\}= \bigcup_{\mathbf r\bmod s}\left\{(U,\mathbf z)\in\Mat_{m,n}(\T):\mathbf z\in\mathcal{D}_\eta(\mathbf b,\mathbf r;U)\right\}.
\end{align*}
On each fiber $U$, $\mathcal E_\eta(\mathbf a)$ and $\mathcal D_\eta(\mathbf b)$ are precisely the corresponding exact cubes and dangerous cubes.
For each $\mathbf a\in \mathcal{Q}_M$ and $\bb\in\mathbb Z^n\setminus\{\bf 0\}$, their measure are given as
\begin{equation*}
\mathcal L_{mn}\left(\mathcal E_\eta(\mathbf a)\right)=q^m\ell_M(\eta)^m\asymp M^m\ell_M(\eta)^m, \ \	
\mathcal L_{mn}\left(\mathcal D_\eta(\mathbf b)\right)=\left(2\eta\psi(|\mathbf b|)\right)^m\asymp\eta^m\psi(|\mathbf b|)^m.
\end{equation*}

Write $S_M(\eta)$ for the total measure of $\mathcal E_\eta(\mathbf a)$ over $\mathbf a\in\mathcal Q_M$:
\begin{equation}\label{eq:24-mass}
S_M(\eta)=\sum_{\mathbf a\in\mathcal Q_M}\mathcal L_{mn}(\mathcal E_\eta(\mathbf a))
\asymp(1-\eta)^m\frac{M^{n+m}}{\log M}\ell_M(\eta)^m.
\end{equation}

We next show that the part of this mass lying in dangerous sets corresponding to denominator vectors of height greater than $M$ is negligible.
\begin{lemma} For every fixed $0<\eta<1$ and all sufficiently large $M$ depending on $\eta$,
\begin{equation}\label{eq:24-overlap}
\sum_{\mathbf a\in\mathcal Q_M}\sum_{\substack{\mathbf b=(\mathbf v,s)\in\Z^n\\ s\ne0,\ |\mathbf b|>M}}\mathcal L_{mn}\left(\mathcal E_\eta(\mathbf a)\cap\mathcal D_\eta(\mathbf b)\right)=o_\eta(S_M(\eta)).
\end{equation}
\end{lemma}

\begin{proof}
We split the argument according to whether $\mathbf a$ and $\mathbf b$ are linearly independent or parallel.
\begin{itemize}
\item When $\mathbf a$ and $\mathbf b$ are linearly independent. It follows from \cite[Chapter~1, Lemma 9]{Spr} that
\[
\mathcal L_{mn}\left(\mathcal E_\eta(\mathbf a)\cap\mathcal D_\eta(\mathbf b)\right)=\mathcal L_{mn}(\mathcal E_\eta(\mathbf a))\mathcal L_{mn}(\mathcal D_\eta(\mathbf b)).
\]
Hence, summing over $\mathbf a\in\mathcal Q_M$ and all linearly independent $\mathbf b=(\mathbf v,s)$ with $s\ne 0$ and $|\mathbf b|>M$,
we obtain
\begin{align*}
&\sum_{\mathbf a\in\mathcal Q_M}\sum_{\substack{\mathbf b=(\mathbf v,s)\in\Z^n\\ s\ne0,\ |\mathbf b|>M\\ \mathbf a,\mathbf b\ \text{linearly independent}}}\mathcal L_{mn}\left(\mathcal E_\eta(\mathbf a)\cap\mathcal D_\eta(\mathbf b)\right)\\
\ll{} & S_M(\eta)\sum_{h>M}h^{n-1}\psi(h)^m=o\bigl(S_M(\eta)\bigr).
\end{align*}

\item When $\mathbf a$ and $\mathbf b$ are parallel. Since every $\mathbf a\in \mathcal{Q}_M$ is primitive, any vector parallel with $\mathbf a$ can be written as $\mathbf b=k\mathbf a$. Since $|\mathbf b|> M$, one has $k\ge 2$.

At first, by Lemma~\ref{l:2-resonant-2}, for any fixed $K$, by asking $M$ sufficiently large, one has that for all $k\le K$,
$$\mathcal E_\eta(\mathbf a)\cap \mathcal D_\eta(k\mathbf a)= \emptyset.$$
So, we need only care about the integers $k$ sufficiently large.

It is trivial that
\[
\mathcal L_{mn}\left(\mathcal E_\eta(\mathbf a)\cap\mathcal D_\eta(k\mathbf a)\right)\le\mathcal L_{mn}\left(\mathcal D_\eta(k\mathbf a)\right)\asymp\eta^m\psi(|k|q)^m.
\]
Hence,
\begin{align*}
\frac{\mathcal L_{mn}\left(\mathcal E_\eta(\mathbf a)\cap\mathcal D_\eta(k\mathbf a)\right)}{\mathcal L_{mn}\left(\mathcal E_\eta(\mathbf a)\right)}
&\ll_{\eta}\left(\frac{\psi(|k|q)}{\psi(M)}\right)^m\ll_{\eta,\varepsilon}|k|^{-m(\tau_\psi-\varepsilon)},
\end{align*}
where we used \eqref{eq:2-regular}. By choosing $K$ sufficiently large, we conclude that	
\[
\sum_{|k|>K}\mathcal L_{mn}\left(\mathcal E_\eta(\mathbf a)\cap\mathcal D_\eta(k\mathbf a)\right)\ll_{\eta,\varepsilon}\delta\, \mathcal L_{mn}\left(\mathcal E_\eta(\mathbf a)\right),
\] where the constant implied in $\ll$ is independent of $\mathbf a$. So summing over $\mathbf a\in\mathcal Q_M$, we get the desired result \eqref{eq:24-overlap}.\qedhere
\end{itemize}
\end{proof}

We now use \eqref{eq:24-overlap} to show that, for most parameters $U$, only a small number of the exact cubes in $\mathscr C_M^\eta(U)$ fail the required goodness condition.

For $U\in\Mat_{m,n-1}(\T)$, let $\mathcal O_M(U)$ denote the total overlap of all exact cubes in $\mathscr C_M^\eta(U)$ with the dangerous cubes corresponding to vectors $\mathbf b=(\mathbf v,s)\in\Z^n, s\ne0,  |\mathbf b|>M.$ More precisely,
\begin{align*}
\mathcal O_M(U)&=\sum_{\mathbf a\in\mathcal Q_M}\ \sum_{\mathbf p\bmod q}\sum_{\substack{\mathbf b=(\mathbf v,s)\in\Z^n\\ s\ne0,\ |\mathbf b|>M}}\ \sum_{\mathbf r\bmod s}\mathcal{L}_m\Big(A_\eta(\mathbf u,\mathbf p,q;U)\cap D_{\eta}(\mathbf b, \mathbf r, U)\Big)\\
&:=\sum_{\mathbf a\in\mathcal Q_M}\ \sum_{\mathbf p\bmod q}\mathfrak B_M^{\eta}(A_\eta(\mathbf u,\mathbf p,q;U)).
\end{align*}
By Fubini's theorem,
\begin{align}\label{eq:24-overlap-fubini}
\int\mathcal O_M(U)\,d\mathcal L_{m(n-1)}(U)&=\sum_{\mathbf a\in\mathcal Q_M}\sum_{\substack{\mathbf b=(\mathbf v,s)\in\Z^n\\s\ne0,\ |\mathbf b|>M}}\mathcal L_{mn}\left(\mathcal E_\eta(\mathbf a)\cap\mathcal D_\eta(\mathbf b)\right).
\end{align}

\begin{proposition}\label{p:2-large} Let $L$, $\xi$ and $\eta$ be as in Proposition~\ref{p:2-tran}.  For every sufficiently large $M$ depending on $L$, $\bm\theta$ and $\eta$, there exists an exceptional set $\mathcal X_3(\xi)\subset\Mat_{m,n-1}(\T)$ of Lebesgue measure at most $\xi$ such that, for every $	U\notin\mathcal X_3(\xi),$ at most
\[
\frac{c_2L^mM^{m+n}}{8\log M}
\]
exact cubes in $\mathscr C_M^\eta(U)$ fail to be $(M,\eta,\delta)$-good with respect to $U$, where $c_2$ is given as in \eqref{eq:21-surv}.
\end{proposition}

\begin{proof}
For an exact cube $Q\in \mathscr C_M^\eta(U)$, if it fails to be $(M,\eta,\delta)$-good with respect to $U$, then
\begin{equation}\label{eq:24-bad-child}
\mathfrak B_M^{\eta}(Q)>\delta\mathcal L_m(Q)=\ell_M(\eta)^m.
\end{equation}

Let
\[
N_>(U):=\#\left\{Q: Q\text{ is not }(M,\eta,\delta)\text{-good with respect to }U\right\}.
\]
Summing \eqref{eq:24-bad-child} over all such exact cubes $Q$, we obtain
\[N_>(U)\, \delta\ell_M(\eta)^m\le\mathcal O_M(U).
\]
Therefore, by \eqref{eq:24-overlap-fubini} and \eqref{eq:24-overlap}, we obtain
\begin{align*}
\int N_>(U)\,d\mathcal L_{m(n-1)}(U)&\le \frac{1}{\delta\ell_M(\eta)^m}\cdot\int \mathcal O_M(U)\,d\mathcal L_{m(n-1)}(U)\\
&\ll \frac{o_\eta(S_M(\eta))}{\delta\ell_M(\eta)^m}\stackrel{\eqref{eq:24-mass}}{\asymp}o_\eta\left(\frac{M^{m+n}}{\delta\log M}\right).
\end{align*}
Hence, by taking $M=M(\eta, \delta, L, \xi)$ sufficiently large,
Markov's inequality is applied to show the set
\[
\mathcal X_3(\xi):=\left\{U\in\Mat_{m,n-1}(\T): N_>(U)>\frac{c_2L^mM^{m+n}}{8\log M}\right\}
\]
has measure at most $\xi$, which is precisely the desired exceptional set.
\end{proof}

\begin{remark}
Again, the exceptional set depends only on the side length of a cube, not on its location.
\end{remark}

\subsection{Proof of Proposition~\ref{p:2-tran}}

This follows directly from Propositions \ref{p:2-low}, \ref{p:2-int}, \ref{p:2-sep} and \ref{p:2-large}.

\section{Completion of the proof: $n\ge 2$}

The proof is similar to $n=1$ case. The only difference is that at every stage during the Cantor subset construction, we need discard a set of fibers $U\in \Mat_{m(n-1)}(\mathbb{T})$ of an arbitrarily small measure. This ensures that there is still a set $\mathcal U_*$ of positive measure of fibers  such that for each $U\in \mathcal U_*$, Proposition \ref{p:2-tran} is valid in every stage of the Cantor subset construction.

Fix $\eta_j=1-2^{-j-1}$ for every $j\ge 0$. Recall that $\Omega$ has positive measure. Let
\begin{equation*}
\tau_j=2^{-j-3}\mathcal L_{m(n-1)}(\Omega),\quad j\ge1.
\end{equation*}
Then, one has
\begin{equation*}
3\sum_{j=1}^{\infty}\tau_j=\frac38\, \mathcal L_{m(n-1)}(\Omega).
\end{equation*}

By increasing $T$ once more if necessary, we may assume, in addition to \eqref{eq:2-no-solu}, that $\T^m$ is $(T,\eta_0,\delta)$-good with respect to every $U\in\Omega.$ This follows from Lemma~\ref{l:2-initial} because of the independence of $T$ on $J$. Let $T_0=T$ and $L_0=1.$ We now choose regular scales $M_1<M_2<\cdots$, recursively.

Suppose that $T_{j-1}$ and $L_{j-1}$ have already been defined. Having fixed $\eta_j$, choose a regular scale $M_j$ sufficiently large so that Proposition~\ref{p:2-tran} applies with
$T=T_{j-1}$, $\eta_0=\eta_{j-1}$, $\eta=\eta_j$, $L=L_{j-1}$ and $\xi=\tau_j$.
Finally, define
\begin{equation*}
T_j=M_j,\quad L_j=\ell_{M_j}(\eta_j).\label{eq:stage-T-L}
\end{equation*}

For each $j\ge1$, let $\mathcal X_j\subset\Omega$ be the exceptional set supplied by Proposition~\ref{p:2-tran} with parameters
$T=T_{j-1}$, $\eta_0=\eta_{j-1}$, $\eta=\eta_j$, $L=L_{j-1}$, $M=M_j$ and $\xi=\tau_j$. Then, we have that
\begin{equation*}
\mathcal L_{m(n-1)}(\mathcal X_j)\le3\tau_j.
\end{equation*}

Define
\begin{equation*}
\mathcal U_*:=\Omega\setminus\bigcup_{j=1}^{\infty}\mathcal X_j.
\end{equation*} which is clearly of positive $\mathcal{L}_{m(n-1)}$ measure. Thus, for every $U\in\mathcal U_*$ and every $j\ge1$, the conclusion of Proposition~\ref{p:2-tran} is available at stage $j$. Moreover, since the exceptional set in that lemma is independent of the position of the parent cube, the same $U$ works simultaneously for all $(T_{j-1},\eta_{j-1},\delta)$-good parent cubes of side length $L_{j-1}$.

Fix $U\in\mathcal U_*.$ The following construction is identically same as Section \ref{s:1-proof}. Starting with
\[\mathcal F_0(U)=\{\T^m\}.\]
By the choice of $T$, the initial cube $\T^m$ is $(T_0,\eta_0,\delta)$-good with respect to $U$.

Suppose that $\mathcal F_{j-1}(U)$ has been constructed, and that every $P\in\mathcal F_{j-1}(U)$ has side length $L_{j-1}$ and is $(T_{j-1},\eta_{j-1},\delta)$-good with respect to $U$. Since $U\notin\mathcal X_j$, Proposition~\ref{p:2-tran} applies simultaneously to every such parent $P$. Then inside each $P$ there are at least
\begin{equation*}
\frac{c_2L_{j-1}^mM_j^{m+n}}{2\log M_j}:=b_j
\end{equation*}
exact cubes of the form
\[
A_{\eta_j}(\mathbf u,\mathbf p,q;U),\quad (\mathbf u,q)\in\mathcal Q_{M_j}.
\]
For each $P\in\mathcal F_{j-1}(U)$, let $\mathcal C_j(P)$ denote the family of the $b_j$ retained exact cubes, and set
\[
\mathcal F_j(U)=\bigcup_{P\in\mathcal F_{j-1}(U)}\mathcal C_j(P).
\]

Therefore, for any $U\in\mathcal U_*$, one can define
\begin{equation*}
K_U=\bigcap_{j=0}^{\infty}\bigcup_{Q\in\mathcal F_j(U)}Q.
\end{equation*}

\begin{proposition} For every $U\in\mathcal U_*$,
\begin{equation*}
\{U\}\times K_U\subset\Exact_{m,n}(\psi,\boldsymbol\theta).
\end{equation*}
\end{proposition}

\begin{proof}
The proof is identically same as the proof of Proposition \ref{p:1-subset} with obvious changes.
\end{proof}

The Hausdorff dimension estimate for $K_U$ is obtained in exactly the same way as in Section \ref{s:1-proof}. It suffices to substitute the corresponding parameters $b_j$, $L_j$ and
 $$r_j\asymp{M_j^{\frac{n}{m}-\rho}(\log M_{j})^{-\frac{1}{m}}}.$$
 into Lemma \ref{l:cantor}. With verbatim argument as in Section~\ref{s:1-proof}, one concludes
\begin{equation*}
\hdim K_U \ge \frac{m+n}{1+\tau_\psi+2\varepsilon}.
\end{equation*}
The slicing lemma (see Lemma \ref{l:2-slicing}) then gives the final conclusion
\[
\hdim \Exact_{m,n}(\psi,\boldsymbol\theta)\ge m(n-1)+\frac{m+n}{1+\tau_\psi}.
\]

\end{document}